\documentclass[a4paper,11pt,twoside]{article}
\usepackage[a4paper]{geometry}
\usepackage[T1]{fontenc}
\usepackage[utf8]{inputenc}
\usepackage[stretch=10]{microtype}
\usepackage{lmodern}
\usepackage[ngerman,main=english]{babel}
\usepackage[autostyle,strict]{csquotes}
\MakeOuterQuote{"}
\usepackage[style=numeric,sortcites=true,giveninits]{biblatex}
\DeclareFieldFormat[article]{volume}{\bibstring{jourvol}\addnbspace #1}
\DeclareSourcemap{%
    \maps[datatype=bibtex]{
        \map[overwrite]{
            \step[fieldsource=shortjournal]
            \step[fieldset=journaltitle,origfieldval]
        }
    }
}
\usepackage{mathtools}
\usepackage{amsthm,amssymb}
\usepackage{graphicx}
\usepackage{enumerate}
\usepackage{algorithm}
\usepackage{algpseudocode}
\usepackage[charsperline=65,linenumbers=false,theme=default]{my-jlcode}%
\usepackage{booktabs}
\usepackage[bookmarksopen,bookmarksnumbered,colorlinks=true,citecolor=blue,pdfpagelabels,unicode]{hyperref}
\usepackage{bookmark}
\newcommand{\MATLAB}{\textsc{Matlab}}
\newcommand{\levy}{L\'{e}vy}
\newcommand{\package}[1]{\texttt{#1}}
\newcommand{\R}{\mathbb{R}}
\newcommand{\differential}{\, \mathrm{d}}
\DeclarePairedDelimiter{\abs}{\lvert}{\rvert}
\DeclarePairedDelimiter{\norm}{\lVert}{\rVert}
\DeclarePairedDelimiterX{\inner}[2]{\langle}{\rangle}{#1,#2}
\DeclarePairedDelimiterXPP{\Normal}[2]{\mathcal{N}}{\lparen}{\rparen}{}{#1,\,#2}

\DeclarePairedDelimiterX{\Set}[1]{\lbrace}{\rbrace}{#1}
\DeclarePairedDelimiterXPP{\Expect}[1]{\mathbb{E}}{\lbrack}{\rbrack}{}{#1}

\DeclareMathOperator{\vecop}{vec}
\DeclareMathOperator{\matop}{mat}

\newcommand{\Prob}{\operatorname{P}}
\newcommand{\Erw}{\mathrm{E}}

\newcommand{\Span}{\operatorname{span}}
\newcommand{\Diag}{\operatorname{\mathrm{diag}}}
\newcommand{\Oo}{\operatorname{\mathcal{O}}}

\newcommand{\Skew}[1]{\mathrm{Skew}_{#1}}
\newcommand{\maxL}{{\ensuremath{\max,L^2}}}
\newcommand{\LF}{{\ensuremath{L^2,\mathrm{F}}}}
\newcommand{\FL}{{\ensuremath{\mathrm{F},L^2}}}
\newcommand{\IterInt}{\mathcal{I}}
\newcommand{\LevyArea}{A}
\newcommand{\app}[1]{\hat#1} %
\newcommand{\Se}{p}		%
\newcommand{\pp}{{(\Se)}}
\newcommand{\trunc}[1]{#1^{\pp}} %
\newcommand{\FS}{\mathrm{FS}}
\newcommand{\Mil}{\mathrm{Mil}}
\newcommand{\Wik}{\mathrm{Wik}}
\newcommand{\MR}{\mathrm{MR}}
\newcommand{\Alg}{\mathrm{Alg}}
\newcommand{\LAFS}{\app{\LevyArea}^{\FS,\pp}}
\newcommand{\LAMil}{\app{\LevyArea}^{\Mil,\pp}}
\newcommand{\LAWik}{\app{\LevyArea}^{\Wik,\pp}}
\newcommand{\LAMR}{\app{\LevyArea}^{\MR,\pp}}
\newcommand{\SFS}{{S}^{\FS,\pp}}
\newcommand{\SMil}{{S}^{\Mil,\pp}}
\newcommand{\SWik}{{S}^{\Wik,\pp}}
\newcommand{\SMR}{{S}^{\MR,\pp}}
\newcommand{\mref}{\mathrm{ref}}
\newtheorem{theorem}{Theorem}

\newtheorem{lemma}[theorem]{Lemma}
\newcommand{\orcidicon}[1]{\href{https://orcid.org/#1}{\includegraphics[height=\fontcharht\font`\B]{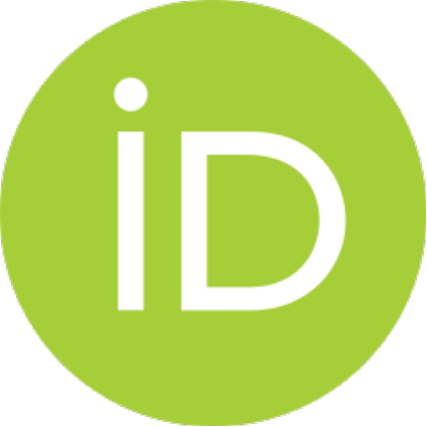}}}
\providecommand{\keywords}[1]{\providecommand{\and}{\textbullet}\textbf{\textit{Keywords ---}} #1}
\begin{document}
\title{An Analysis of Approximation Algorithms for Iterated Stochastic Integrals
     and a Julia and \MATLAB{} Simulation Toolbox}
\author{Felix Kastner\thanks{Funding and support by the Graduate School for Computing in Medicine and Life Sciences funded by Germany's Excellence Initiative [DFG GSC 235/2] is gratefully acknowledged.} \thanks{e-mail: kastner@math.uni-luebeck.de}\ \orcidicon{0000-0003-0504-6211}\ and Andreas Rößler\thanks{e-mail: roessler@math.uni-luebeck.de}\ \orcidicon{0000-0002-2740-2697}
    \medskip\\
    \small{Institute of Mathematics, Universität zu Lübeck,} \\[-1mm]
    \small{Ratzeburger Allee 160, 23562 Lübeck, Germany} %
}
\date{January 20, 2023}
\maketitle
\begin{abstract}
For the approximation and simulation of twofold iterated
stochastic integrals and the corresponding L\'{e}vy areas
w.r.t.\ a multi-dimensional Wiener process, we review four algorithms
based on a Fourier series approach. Especially, the
very efficient algorithm due to Wiktorsson and a newly
proposed algorithm due to Mrongowius and R\"o{\ss}ler are considered.
To put recent advances into context, we analyse the four Fourier-based
algorithms in a unified framework to highlight differences and similarities
in their derivation. A comparison of theoretical properties is complemented
by a numerical simulation that reveals the order of convergence for each
algorithm. Further, concrete instructions for the choice of the optimal
algorithm and parameters for the simulation of solutions for
stochastic (partial) differential equations are given.
Additionally, we provide advice for an efficient implementation of
the considered algorithms and incorporated these insights into an open
source toolbox that is freely available for both Julia and \MATLAB{}
programming languages. The performance of this toolbox is analysed by
comparing it to some existing implementations, where we observe a
significant speed-up.

\end{abstract}
\newpage
\section{Introduction}
\label{sec:intro}
\begin{quote}
    In the future, we foresee that the use of area integrals when simulating strong solutions to SDEs will become as automatic as the use of random numbers from a normal distribution is today. After all, once a good routine has been developed and implemented in numerical libraries, the ordinary user will only need to call this routine from each program and will not need to be concerned with the details of how the routine works.

    \hfill --- J. G. Gaines and T. J. Lyons (1994) \cite{MR1284705}
\end{quote}

Even though Wiktorsson published an efficient algorithm for the simulation of iterated integrals in 2001, this prediction by Gaines and Lyons has still, almost thirty years later, not come true yet.
Among the twenty SDE simulation packages we looked at, only four have implemented Wiktorsson's algorithm.
This suggests that either Wiktorsson's paper is not as well known as it should be or most authors of SDE solvers are not willing to dive into the intricacies of L\'{e}vy area simulation.
This paper is an attempt to fix this.

In many applications where stochastic (ordinary) differential equations
(SDEs) or stochastic partial differential equations (SPDEs)
are used as a mathematical model
the exact solution of these equations can not be calculated explicitly.
Therefore, numerical approximations have to be computed in these cases. Often
numerical integrators solely based on increments of the driving Wiener process or
$Q$-Wiener process like Euler-type schemes are applied because
they can be easily implemented. In general, these schemes attain only low
orders of strong convergence, see, e.g., \cite{MR1214374,MR1335454}.
This is because the best strong convergence rate we can expect to get is
in general at most~$1/2$ if only finitely many observations of the driving
Wiener process are used by a numerical scheme, see, e.g., Clark and
Cameron \cite{MR609181} and Dickinson \cite{MR2352954}.
One exception to this is the case when the diffusion
satisfies a so-called commutativity condition
\cite{MR1214374,MR3842926,MR1335454,MR2669396}.
However, in many applications this commutativity condition is not satisfied.

In order to obtain higher order numerical methods for SDEs and SPDEs in general, one needs to incorporate iterated stochastic integrals as they naturally appear in stochastic Taylor expansions of the solution process.
Since the distribution of these iterated stochastic
integrals is not known explicitly in case of a multi-dimensional driving
Wiener process, efficient algorithms for their approximation
are needed to achieve some higher order of strong convergence
compared to Euler-type schemes.
Thus, in general, to approximate or to simulate solutions for SDEs or SPDEs efficiently the use of iterated stochastic integrals is essential.

Up to now, there exist only a few approaches for the approximation
of iterated stochastic integrals and the corresponding L\'{e}vy areas
in $L^2$-norm or in the strong sense.
After \levy\ defined the eponymous stochastic area integral in
\cite{MR0044774}, the first approximation algorithms appeared in
the 80s \cite{Liske1982,MR1335454}.
These are based on a Fourier series expansion of the Wiener process
and have an error of order $O(c^{-1/2})$ with
$c$ denoting the computational cost, considered as the number
of random variables used.
The first higher order algorithm was presented by Ryd\'{e}n and
Wiktorsson in \cite{MR1807367}.
Here, the method~B' from that paper has an error of order $O(c^{-1})$
but is only applicable to two-dimensional Wiener processes.
In the same year, Wiktorsson published his groundbreaking algorithm
of order $O(c^{-1})$ which also works for higher-dimensional Wiener
processes \cite{MR1843055} but has a worse error constant.
Twenty years later, Mrongowius and R\"o{\ss}ler~\cite{Mrongowius2021} recently proposed a variation of Wiktorsson's algorithm that has an improved error constant and is easier to implement.
It turns out that in the two-dimensional setting it even beats the method~B', that is, the error constant is by a factor~$1/{\sqrt{2}}$ smaller compared to method~B' for the same amount of work.

These higher order algorithms allow to effectively use strong order~$1$ schemes for
SDEs and SPDEs with non-commutative noise.
Taking into account the computational effort of SDE and SPDE integrators, one
obtains the so-called effective order of convergence by considering errors versus
computational cost, see also \cite{MR3842926,MR2669396}.
For example, combining a strong order~$1$ SDE integrator, like the Milstein scheme
or a stochastic Runge--Kutta scheme proposed in~\cite{MR2669396}, with a \levy\
area algorithm of order $O(c^{-1/2})$ achieves an effective order of convergence
that is only~$1/2$.
This is the same as the effective order of convergence for the classical
Euler--Maruyama scheme.
But, using instead a higher order \levy\ area algorithm of order $O(c^{-1})$
results in a significantly higher effective order of convergence that is~$2/3$.
Similar improvements can be achieved for SPDEs, see also
\cite{vonHallern2020a,vonHallern2020,MR2669396,MR1843055} for a detailed
discussion.

The algorithms that we focus on in the present paper are based on a trigonometric Fourier series expansion of the corresponding Brownian bridge process.
To be precise, we restrict our considerations to the naive truncated Fourier series approximation, the improved truncated Fourier series algorithm proposed by Milstein~\cite{MR1335454}, see also \cite{MR1178485,MR1214374}, and the higher order variants by Wiktorsson~\cite{MR1843055} and Mrongowius and Rößler~\cite{Mrongowius2021}.
As we will see, the latter two incorporate an additional approximation of the truncated terms to achieve an error of order $O(c^{-1})$.

Note that there also exist further variants of Fourier series based algorithms using different bases \cite{Liske1982,Kuznetsov2018,Kuznetsov2019,Foster2022}.
These generally achieve the same convergence order as the Milstein algorithm but with a slightly worse constant.
This matches Dickinson's result, that the Milstein algorithm is asymptotically optimal in the class of algorithms that only use linear information about the underlying Wiener process \cite{MR2352954}.
There is also work extending some of the mentioned algorithms to the infinite dimensional Hilbert space setting important for SPDEs \cite{MR3949104}.

Next to these Fourier series based algorithms, there exist also different approaches.
See, e.g., \cite{MR1284705,MR1807367,Malham2014,Stump2005} for simulation
algorithms or \cite{Davie2014,Flint2015} for approximation in a Wasserstein metric.
However, to the best of our knowledge, these algorithms either lack a $L^2$-error
estimate or come with additional assumptions on the target SDE and are thus not
suitable for the general strong approximation of SDEs.
Moreover, for the approximation of L\'{e}vy areas driven by fractional Brownian
motion, we refer to \cite{Neuenkirch2016,Neuenkirch2010}.

The aims of this paper are twofold.
On the one hand, we give an introduction to the different Fourier series based
algorithms under consideration and emphasize their similarities and differences.
A special focus lies on the analysis of the computational complexity for these
algorithms that is important in order to finally justify their application for SDE or
SPDE approximation and simulation problems.
On the other hand, we provide an efficient implementation of these algorithms.
This is essential for a reasonable application in the first place.
The result is a new Julia and \MATLAB{} software toolbox for the simulation of
twofold iterated stochastic integrals and their L\'{e}vy areas, see
\cite{levyarea-jl-zenodo,levyarea-m-zenodo}.
These two packages make it feasible to use higher order approximation schemes,
like Milstein or stochastic Runge--Kutta schemes, for non-commutative SDEs
and SPDEs.

The paper is organized as follows:
In Section~\ref{sec:theory} we give a brief introduction to the theoretical background
of iterated stochastic integrals with the corresponding L\'{e}vy areas and the
Fourier series expansion which builds the basis for the approximation algorithms
under consideration.
Based on this preliminary section, we detail the derivation of the four algorithms
under consideration in Section~\ref{sec:algorithms}.
Besides the theoretical foundations, we focus on the efficient implementation of
the four algorithms.
In Section~\ref{sec:error-estimates}, error estimates for each algorithm are
gathered for direct comparison.
Here, it is worth noting that some slightly improved result has been found for
Wiktorsson's algorithm.
Furthermore, we review the application of these algorithms to the numerical
simulation of SDEs.
Continuing in Section~\ref{sec:theo-comp}, we give a detailed analysis of the
computational complexity and determine the optimal algorithm for a range of
parameters.
In addition, a numerical simulation confirms the theoretical orders of convergence
from the previous section.
The paper closes in Section~\ref{sec:software} with the description of the newly
developed software toolbox and a run-time benchmark against currently available
implementations.

\section{Theoretical foundations for the simulation of iterated stochastic integrals}
\label{sec:theory}
Here, we give a brief introduction to twice iterated stochastic
integrals in terms of Wiener processes (also called Brownian motions)
and their relationship to L\'{e}vy areas. In addition, the
infinite dimensional case of a $Q$-Wiener process is briefly mentioned as well.
Based on these fundamentals, the well known Fourier series expansion of
the Brownian bridge process is presented, which builds the basis for
all algorithms in this paper.
\subsection{Iterated stochastic integrals and the L\'{e}vy area}
\label{subsec:iterint-levyarea}
Let $(\Omega, \mathcal{F}, \Prob)$ be a complete probability space and
let $(W_t)_{t \in [0,T]}$ be an $m$-dimensional Wiener process for some
$0<T<\infty$ with $(W_t)_{t \in [0,T]} = ((W_t^1, \ldots,
W_t^m)^\top)_{t \in [0,T]}$, i.e.,
the components $(W_t^i)_{t \in [0,T]}$, $ i=1,\ldots,m $, are independent scalar Wiener processes.
Further, let $\| \cdot \|$ denote the Euclidean norm
if not stated otherwise and let $\| X \|_{L^2(\Omega)} = \Erw\big[ \| X \|^2\big]^{1/2}$
for any $X \in L^2(\Omega)$ in the following.
We are interested in simulating the increments
$\Delta W_{t_0,t_0+h}^i = W_{t_0+h}^i-W_{t_0}^i$ and
$\Delta W_{t_0,t_0+h}^j =  W_{t_0+h}^j-W_{t_0}^j$ of the Wiener process
together with the twice iterated stochastic integrals
\begin{equation} \label{eq:iteratedintegrals-t0-h}
    \IterInt_{(i,j)}(t_0,t_0+h)
    = \int_{t_0}^{t_0+h} \!\! \int_{t_0}^{s} \differential W_r^i\differential W_s^j
\end{equation}
for some $0 \leq t_0< t_0+h \leq T$ and $1 \leq i, j \leq m$.
Note that for the Wiener process $(\hat{W}_t)_{t \in [0,h]}$ with
$\hat{W}_t = W_{t_0+t} - W_{t_0}$, see \cite[Ch. 2, Lem. 9.4]{MR1121940}, it
holds that
\begin{equation} \label{eq:iteratedintegrals-0-h}
    \int_{0}^{h} \!\! \int_{0}^{s} \differential\hat{W}_r^i \differential \hat{W}_s^j
    = \int_{t_0}^{t_0+h} \!\! \int_{t_0}^{s} \differential W_r^i \differential W_s^j
\end{equation}
for $1 \leq i,j \leq m$.
Moreover, due to the time-change formula for stochastic integrals
\cite[Ch. 3, Prop. 4.8]{MR1121940} one can show that for the
scaled Wiener process $( \tilde{W}_t )_{t \in [0,1]}$ with $\tilde{W}_t
= \frac{1}{\sqrt{h}} \hat{W}_{h t}$, see \cite[Ch. 2, Lem. 9.4]{MR1121940}, it holds
\begin{equation} \label{eq:iteratedintegrals-time-scaling}
    h  \int_{0}^{1}\!\!\int_{0}^{s}\differential\tilde{W}_r^i\differential\tilde{W}_s^j
    = \int_{0}^{h}\!\!\int_{0}^{s}\differential \hat{W}_r^i\differential \hat{W}_s^j \, .
\end{equation}
As a result of this, without loss of generality we restrict our considerations to
the case $t_0=0$ in the following and denote
\begin{equation}
	\IterInt_{(i,j)}(h) = \int_{0}^{h} \!\! \int_{0}^{s} \differential {W}_r^i
	\differential {W}_s^j \, .
\end{equation}
Further, let $\IterInt(h) = ( \IterInt_{(i,j)}(h) )_{1 \leq i,j \leq m}$ be
the $m \times m$-matrix containing all iterated stochastic integrals.

In some special cases, one may circumvent the simulation of the iterated
stochastic integrals by making use of the relationship
\begin{align} \label{eqn:commu-noise-relation}
	\frac{1}{2} \big( \IterInt_{(i,j)}(h) + \IterInt_{(j,i)}(h) \big)
	= \frac{1}{2} \, W_{h}^i \, W_{h}^j
\end{align}
for $1 \leq i < j \leq m$, see, e.g., \cite{MR1214374}, where the left hand side
represents the symmetric part
of the iterated stochastic integrals that can be expressed by the corresponding
increments of the Wiener process.
The right hand side of \eqref{eqn:commu-noise-relation} can be easily
simulated since the random variables $W_{h}^i \sim \Normal{0}{h}$
are i.i.d.\ distributed for $1 \leq i \leq m$.
In case $i=j$ we can calculate explicitly that
\begin{equation} \label{eqn-I-ii-formula}
	\IterInt_{(i,i)}(h) = \frac{1}{2} \big( ( W_{h}^i )^2 - h \big)
\end{equation}
which follows from the It{\^o} formula, see, e.g., \cite{MR1121940}.

The problem to simulate realizations of iterated stochastic
integrals is directly related to the simulation of the corresponding so-called
L\'{e}vy area \cite{MR0044774} $\LevyArea_{(i,j)}(h)$ defined as the skew-symmetric part of the iterated integrals
\begin{equation}
	\LevyArea_{(i,j)}(h) = \frac{1}{2} \big( \IterInt_{(i,j)}(h) - \IterInt_{(j,i)}(h) \big)
\end{equation}
for $1 \leq i,j \leq m$. We denote by $\LevyArea(h) = (\LevyArea_{(i,j)}(h))_{1 \leq i,j \leq m}$ the $m \times m$ matrix of all L\'{e}vy areas.
Thus it holds that $ \LevyArea(h) = -\LevyArea(h)^\top $ and $ \LevyArea_{(i,i)} = 0 $ for all $ 1\leq i\leq m $.
Due to \eqref{eqn:commu-noise-relation}, it follows that
$\mathcal{I}_{(i,j)}(h) = \frac{1}{2} \, W_{h}^i \, W_{h}^j + \LevyArea_{(i,j)}(h)$
for $i \neq j$. Now, the difficult part is to simulate
$\mathcal{I}_{(i,j)}(h)$ for $i \neq j$ because the distribution of the
corresponding L\'{e}vy area $\LevyArea_{(i,j)}(h)$ is not known.

For the infinite dimensional setting as it is the case for SPDEs of evolutionary type,
we follow the approach in \cite{MR3949104}. Therefore, we consider a
$U$-valued and in general infinite dimensional $Q$-Wiener process
$(W_t^Q)_{t \in [0,T]}$
taking values in some separable real Hilbert space $U$. Let $(\eta_i)_{i \in
\mathbb{N}}$ denote the eigenvalues of the trace class, non-negative and symmetric covariance operator
$Q \in L(U)$ w.r.t.\ an orthonormal basis (ONB) of eigenfunctions
$(e_i)_{i \in \mathbb{N}}$ such that $Q e_i = \eta_i e_i$ for all $i \in \mathbb{N}$.
We assume that $\eta_i>0$ for $i=1, \ldots, m$. Then,
the orthogonal projection of $W_t^Q$ to the $m$-dimensional subspace
$U_m = \Span \{ e_i : 1 \leq i \leq m \}$ of $U$ is given by
\begin{align*}
	W_t^{Q,m} = \sum_{i=1}^m \sqrt{\eta_i} \, e_i \, W_t^i , \quad t \in [0,T],
\end{align*}
where $(W_t^i)_{t \in [0,T]}$, $1 \leq i \leq m$, are independent
scalar Wiener processes, see, e.g., \cite{MR3236753,MR2329435}.
The corresponding finite-dimensional covariance
operator $Q_m$ is then defined as $Q_m = \Diag (\eta_1, \ldots, \eta_m) \in
\mathbb{R}^{m \times m}$ and the vector
$(\langle W_h^{Q,m}, e_i \rangle)_{1 \leq i \leq m}$
is multivariate $\Normal{0}{hQ_m}$ distributed for $0 < h \leq T$.
We want to simulate the Hilbert space valued iterated stochastic integral
\begin{align}\label{eq:iteratedintegrals-Q}
	\int_0^h \Psi \int_0^s \Phi \, \mathrm{d}W_u^{Q,m} \, \mathrm{d}W_s^{Q,m}
	= \sum_{i,j=1}^m \mathcal{I}_{(i,j)}^Q(h) \, \Psi( \Phi e_i,e_j)
\end{align}
where
$\mathcal{I}_{(i,j)}^Q(h) = \int_0^h \int_0^s \langle \mathrm{d}W_u^{Q},
e_i \rangle_U \, \langle \mathrm{d}W_s^{Q}, e_j \rangle_U$ and where $\Psi$
and $\Phi$
are some suitable linear operators, see \cite{MR3949104} for details. Let
$\mathcal{I}^{Q_m}(h) = ( \mathcal{I}_{(i,j)}^{Q}(h) )_{1 \leq i,j \leq m}$ denote the
corresponding $m \times m$-matrix. Since $\mathcal{I}^{Q_m}(h) = Q_m^{1/2}
\mathcal{I}(h) Q_m^{1/2}$, this random matrix can be expressed by a
transformation of the $m \times m$-matrix $\mathcal{I}(h)$. Therefore, we
mainly concentrate on the approximation of $\mathcal{I}_{(i,j)}(h)$ for $1 \leq
i,j \leq m$ in the following and refer to \cite{MR3949104} for more details
and error estimates in the infinite dimensional setting.
\subsection{The Fourier series approach}
\label{subsec:fourier-series}
We start by focusing on the Fourier series expansion of the Brownian bridge
process, see \cite{MR1214374,MR1178485,MR1335454}.
The integrated Fourier series expansion is the basis for all algorithms considered
in Section~\ref{sec:algorithms}.
Given an $m$-dimensional Wiener process $(W_t)_{t \in [0,h]}$, we
consider the corresponding tied down Wiener process $\left(W_t-\tfrac{t}{h}
W_h \right)_{t \in [0,h]}$, which is also called a Brownian bridge process,
whose components can be expanded into a Fourier series which results in
\begin{equation} \label{eqn-Fourier-BB}
    W_t^i = \frac{t}{h} W_h^i + \frac{1}{2} a_0^i + \sum_{r=1}^{\infty} \bigg( a_r^i
    \cos \bigg( \frac{2 \pi r}{h} t \bigg) + b_r^i \sin \bigg( \frac{2 \pi r}{h} t \bigg)
    \bigg) \quad \Prob\text{-a.s.}
\end{equation}
with random coefficients
\begin{align*}
    a_r^i = \frac{2}{h} \int_{0}^{h} \Big( W_s^i - \frac{s}{h} W_h^i \Big)
    \cos \bigg( \frac{2 \pi r}{h} s \bigg) \differential s \\
    \shortintertext{and}
    b_r^i = \frac{2}{h} \int_{0}^{h} \Big( W_s^i - \frac{s}{h} W_h^i \Big)
    \sin \bigg( \frac{2 \pi r}{h} s \bigg) \differential s
\end{align*}
for $1 \leq i \leq m$ and $t \in [0,h]$.
Using the distributional properties of the Wiener integral it easily follows that
the coefficients are Gaussian random variables with
\begin{equation} \label{eq:fourier-coeffs-distribution}
    a_0^i \sim \Normal{0}{\tfrac{1}{3}h} , \quad
    a_r^i \sim \Normal{0}{\tfrac{1}{2 \pi^2 r^2} h} \quad
    \text{ and } \quad
    b_r^i \sim \Normal{0}{\tfrac{1}{2 \pi^2 r^2} h} \, .
\end{equation}
From the boundary conditions it follows that $a_0^i = -2 \sum_{r=1}^{\infty}
a_r^i$. One can easily prove that the random variables $W_h^q$, $a_k^i$ and
$b_l^j$ for $i,j,q \in \{1, \ldots, m\}$ and $k,l \in \mathbb{N}$ are all independent,
while each $a_0^i$ depends on $a_r^i$ for all $r \in \mathbb{N}$.
Using this representation L\'{e}vy was the first to derive the following series
representation of what is now called L\'{e}vy area \cite{MR0044774} denoted
as $ \LevyArea_{(i,j)}(h) = \frac{1}{2}(\IterInt_{(i,j)}(h)-\IterInt_{(j,i)}(h)) $.
By integrating \eqref{eqn-Fourier-BB} with respect to the Wiener process
$(W_t^j)_{t \in [0,h]}$ and following the representation in
\cite{MR1214374,MR1178485}, we get the representation
\begin{align} \label{eqn-Expansion-iterated-integral}
    \mathcal{I}_{(i,j)}(h) &= \frac{1}{2} W_h^i W_h^j - \frac{1}{2} h \, \delta_{i,j}
    + A_{(i,j)}(h)
\end{align}
with the L\'{e}vy area
\begin{align} \label{eq:expansion-levy-area}
    A_{(i,j)}(h) &= \pi \sum_{r=1}^{\infty} r \,
    \Big( a_r^i \Big( b_r^j - \frac{1}{\pi r} W_h^j \Big) - \Big( b_r^i
    - \frac{1}{\pi r} W_h^i \Big) a_r^j \Big)
\end{align}
for $i, j \in \{1, \ldots, m\}$. This series converges in $L^2(\Omega)$, see, e.g.,
\cite{MR1214374,MR1178485,MR1335454}.
In the following, we consider the whole matrix of all iterated stochastic integrals
\begin{equation} \label{eq:Iter-Ito-Int-Levy}
	\mathcal{I}(h) = \frac{1}{2} \big( W_h W_h^\top - h \, I_m \big) + A(h)
\end{equation}
with identity matrix $I_m$ and with the L\'{e}vy area matrix $A(h)$ which
can be written as
\begin{equation} \label{eq:expansion-levy-area-split}
	A(h) = \sum_{r=1}^{\infty} \big( W_h a_r^\top-a_r W_h^\top \big)
	+ \pi \sum_{r=1}^{\infty} r \, \big( a_r b_r^\top - b_r a_r^\top \big)
\end{equation}
where $a_r = (a_r^i)_{1 \leq i \leq m}$ and $b_r = (b_r^i)_{1 \leq i \leq m}$.
The basic approach to the approximation of the L\'{e}vy area consists in
truncating the series \eqref{eq:expansion-levy-area-split}. Then, one
may additionally approximate some or even all of the arising rest terms
in order to improve the approximation.
If both sums in \eqref{eq:expansion-levy-area-split} are truncated at a
point $\Se \in \mathbb{N}$, we denote the truncated Lévy area matrix
as $\trunc{A}(h)$ and the rest terms as $ \trunc{R_1}(h) $ and $ \trunc{R_2}(h)$
where
\begin{align}
	\trunc{A}(h) &\coloneqq \pi \sum_{r=1}^{\Se} r \,
	\left( a_r \left( b_r-\frac{1}{\pi r}W_h \right)^\top
	- \left( b_r-\frac{1}{\pi r}W_h \right) a_r^\top \right) ,
	\label{eq:levy-area-trunc} \\
	\trunc{R_1}(h) &\coloneqq \sum_{r=\Se+1}^\infty W_h a_r^\top
	- a_r W_h^\top , \label{eq:a0-rest} \\
	\trunc{R_2}(h) &\coloneqq \pi \sum_{r=\Se+1}^\infty r \, \left( a_r b_r^\top
	- b_r a_r^\top \right) , \label{eq:ab-rest}
\end{align}
such that $ A(h) = \trunc{A}(h) + \trunc{R_1}(h) + \trunc{R_2}(h)$.
These terms build the basis for the simulation algorithms of the iterated
stochastic integrals that will be discussed in the following sections.
\subsection{Approximation vs. simulation}
\label{subsec:approx-vs-simulation}
In this article, we restrict our considerations to the simulation problem of
iterated stochastic integrals, which has to be distinguished from the
corresponding approximation problem. Although all algorithms under
consideration can also be applied for the approximation of iterated
stochastic integrals, we don't go into details here and refer to, e.g.,
\cite{Mrongowius2021,MR1807367,MR1843055}.
It is worth noting that for the approximation problem where
one is interested to approximate some fixed realization of the iterated
stochastic integrals together with the realization of the increments of
the Wiener process, one needs some information about the realization
of the path of the involved Wiener process. What kind of information is
needed depends on the approximation algorithm to be applied. E.g., for
a truncated Fourier series algorithm one may assume to have access to
the realizations of the first $n$ Fourier coefficients $a_0^i$, $a_r^i$ and
$b_r^i$ for $i \in \{1, \ldots, m\}$ and $1 \leq r \leq n$ together with the
realizations of the increments of the Wiener process. Then, using this
information the goal is to calculate an
approximation for the iterated stochastic integrals such that, e.g., the strong or
$L^2$-error is as small as possible.
However, for many problems one does not have this information about
realizations but rather tries to model uncertainty by, e.g., random processes.
In such situations, one is often interested in the simulation of possible
scenarios that may occur. Therefore, no prescribed information is given
a priori and one has to generate realizations of all involved random
variables. This can be done easily whenever one can sample from some
well known distribution. However, if the distribution of some random variable
such as an iterated stochastic integral is not known, one has to sample
from an approximate distribution. In contrast to weak or distributional
approximation, one still has to care about, e.g., the strong or $L^2$-error
between the simulated approximate realization and a corresponding
hypothetical exact realization. This means that for each simulated realization
one can find at least one corresponding posteriori realization such that
a strong or $L^2$-error estimate with some prescribed precision is always
fulfilled. Since in many or even nearly all practical situations information
about uncertainty is not explicitly available, stochastic simulation algorithms
are frequently used and therefore of course most relevant.
\subsection{Simulating with a prescribed precision}
\label{subsec:simu-precision}
In the following, we consider the problem of simulating realizations
of the iterated stochastic integrals
$\mathcal{I}(h) = ( \mathcal{I}_{(i,j)}(h) )_{1 \leq i,j \leq m}$
provided that the increments $W(h)=(W_h^i)_{1 \leq i \leq m}$ are
given. The simulation of realizations of the stochastically independent
increments $W_h^i$ is straightforward since their distribution $W_h^i \sim
\Normal{0}{h}$ is known and sampling from a Gaussian distribution can be
done easily in principle. However, for the simulation of the iterated stochastic
integrals the situation is different because the random variables
$\mathcal{I}_{(i,j)}(h)$ are not independent and as yet we don't know
their joint (conditional) distribution explicitly. Therefore, we need some
approximate sampling method.

There exist different approaches to do this, and we restrict
our attention to sampling algorithms where the $L^2$-error can be
controlled. This is important, especially if such a simulation algorithm is
combined with a numerical SDE or SPDE solver for the simulation of
root mean square approximations with some prescribed precision.

To be precise, given some precision $\varepsilon >0$ assume that our
simulation algorithm gives us by sampling a realization of the increments
$\hat{W}_h(\omega)$ and a realization of our approximation
$\hat{\mathcal{I}}^{\varepsilon}(h)(\omega)$ for the iterated stochastic integrals
for some $\omega \in \Omega$.
Then, for all algorithms under consideration we can assume that there
exist copies $\tilde{W}_h$ and $\tilde{\mathcal{I}}(h)$ of the random variables
$W_h$ and $\mathcal{I}(h)$ on the probability space $(\Omega, \mathcal{F},
\Prob)$ such that $\tilde{W}_h$ together with $\tilde{\mathcal{I}}(h)$
have the same joint distribution as $W_h$ together with $\mathcal{I}(h)$
and such that $\hat{W}_h(\omega)$, $\hat{\mathcal{I}}^{\varepsilon}(h)(\omega)$
are an approximation of $\tilde{W}_h(\omega)$,
$\tilde{\mathcal{I}}(h)(\omega)$ for $\Prob$-almost all $\omega \in \Omega$
in the sense that $\hat{W}_h = \tilde{W}_h$ $\Prob$-a.s.\ and
\begin{align*}
	\big\| \tilde{\mathcal{I}}_{(i,j)}(h)
	- \hat{\mathcal{I}}^{\varepsilon}_{(i,j)}(h) \big\|_{L^2(\Omega)} \leq \varepsilon
\end{align*}
for $1 \leq i,j \leq m$.
This is also known as a kind of coupling, see, e.g., \cite{MR1807367}
for details.
For simplicity of notation, we
do not distinguish between the random variables $W_h$, $\mathcal{I}(h)$
and their copies $\tilde{W}_h$, $\tilde{\mathcal{I}}(h)$ in the following.
\subsection{Relationship between different error criteria}
\label{subsec:relation-matrix-norms}
Usually, we are interested in simulating the whole matrix
$\mathcal{I}(h)$ with some prescribed precision. Therefore, it is often
appropriate to consider the $L^2$-error with respect to some suitable
matrix norm depending on the problem to be simulated. Here, we
focus our attention to the error with respect to the norms
\begin{align*}
	\| M \|_\maxL &=
	\max_{i,j} \, \| M_{i,j} \|_{L^2(\Omega)} &\quad \quad
	\| M \|_{L^2, \max} &
	= \Erw \big[ \max_{i,j} \, \abs{M_{i,j}}^2 \big]^{1/2}
	\\
	\| M \|_\FL &= \big( \sum_{i,j} \Erw \big[ \abs{M_{i,j}}^2 \big]
	\big)^{1/2}
	&\quad \quad
	\| M \|_\LF &= \Erw \big[ \|M\|_{\mathrm{F}}^2 \big]^{1/2}
\end{align*}
for some random matrix $M = (M_{i,j}) \in L^2(\Omega, \R^{m\times n})$
where $\| \cdot \|_{\mathrm{F}}$ denotes the Frobenius norm.
Clearly, the \FL-norm and the \LF-norm coincide in the $ L^2 $-setting.
In general, these norms are equivalent with
\begin{equation*}
	\| M \|_{\max, L^2} \leq \| M \|_{L^2, \max} \leq \| M \|_\LF
	\leq \sqrt{mn} \| M \|_{\max, L^2} \, .
\end{equation*}
In practice we mostly need the \maxL-norm and the
\LF-norm as these are the natural expressions that
show up if numerical schemes for the approximation of SDEs and SPDEs are
applied, respectively, see \cite[Cor. 10.6.5]{MR1214374} and \cite{MR3949104}.

Since the symmetric part of the iterated stochastic integrals can be calculated
exactly, we concentrate on the skew-symmetric L\'{e}vy area part.
Let $\Skew{m} \subset \R^{m\times m}$ denote the vector space of real,
skew-symmetric matrices.
Then, $\mathcal{I}(h) - \hat{\mathcal{I}}^{\varepsilon}(h)
= A(h) - \hat{A}^{\varepsilon}(h)$ where $\hat{A}^{\varepsilon}(h)$
is the corresponding approximation of the L\'{e}vy area matrix $A(h)$.
Note that the elements of the L\'{e}vy area matrix are identically
distributed and that $A(h), \hat{A}^{\varepsilon}(h) \in \Skew{m}$.
Therefore, the following more precise relationship between the matrix norms under
consideration is used in the following.
\begin{lemma} \label{lem:norm-aequivalenz}
	Let $ A \in L^2(\Omega,\Skew{m})$, i.e., $A(\omega)=-A(\omega)^\top$ for $ \Prob $-a.a.\ $ \omega\in\Omega $ and $ A_{i,j}\in L^2(\Omega) $ for all $ 1\leq i, j\leq m $. Then, it holds
	\begin{align}
		\| A \|_\LF &\leq \sqrt{m^2-m \mathstrut} \, \| A \|_\maxL \, ,
		\label{eq-Norm-Lemma-FL2-maxL2} \\
		\sqrt{2} \, \| A \|_{L^2, \max} &\leq \| A \|_\LF \, .
		\label{eq-Norm-Lemma-L2max-FL2}
	\end{align}
    Additionally, equality holds in \eqref{eq-Norm-Lemma-FL2-maxL2} if the elements $ A_{i,j} $ for $ 1\leq i, j\leq m $ with $ i\neq j $ have identical second absolute moments.
\end{lemma}
\begin{proof}
	Since $A \in L^2(\Omega,\Skew{m})$ it holds $ A_{i,i} = 0 $ $ \Prob $-a.s.\ for all
	$1 \leq i \leq m$. Then, it follows
	\begin{align*}
		\| A \|_\LF^2 &= \sum_{i,j}
		\Erw \big[ | A_{i,j} |^2 \big]
		\leq \sum_{\substack{i,j\\ i \neq j}}
		\max_{k,l} \Erw \big[ |A_{k,l}|^2 \big]
		= (m^2-m) \, \| A \|_\maxL^2 \, .
	\end{align*}
    If $ \Erw \big[ | A_{i,j} |^2 \big] $ is constant for all $ i\neq j $, then we have $ \Erw \big[ | A_{i,j} |^2 \big] = \max_{k,l} \Erw \big[ |A_{k,l}|^2 \big] $ for $ i\neq j $ and equality holds.
    
	Finally, the second inequality follows easily due to
	\begin{align*}
		\| A \|_{L^2,\max}^2 &= \Erw \big[ \max_{i,j} |A_{i,j}|^2 \big]
		= \Erw \big[ \max_{\substack{i,j \\ i<j}} |A_{i,j}|^2 \big]
		\leq \Erw \Big[ \sum_{\substack{i,j \\ i<j}} |A_{i,j}|^2 \Big]
		= \frac{1}{2} \Erw \Big[ \sum_{i,j} |A_{i,j}|^2 \Big] \, .
	\end{align*}
\end{proof}
\noindent
Taking into account the relationship between the described matrix norms,
we concentrate on the \maxL-error of the simulated L\'{e}vy area
\begin{equation*}
	\| A(h) - \hat{A}^{\varepsilon}(h) \|_\maxL \leq \varepsilon
\end{equation*}
for $\varepsilon>0$ because we can directly convert between the different
norms according to Lemma~\ref{lem:norm-aequivalenz}.
\section{The algorithms for the simulation of L\'{e}vy areas}
\label{sec:algorithms}
For the approximation of the iterated stochastic integrals $\mathcal{I}(h)$
in \eqref{eq:Iter-Ito-Int-Levy}, one has to approximate the corresponding
L\'{e}vy areas $A(h)$ in \eqref{eq:expansion-levy-area-split}.
A first approach might be to truncate the expansion of the L\'{e}vy areas at some
appropriate index $\Se$ as in \eqref{eq:levy-area-trunc} such that the truncation
error is sufficiently small.
This is the main idea for the first algorithm called Fourier algorithm which is
presented in Section~\ref{subsec:alg-fourier-series}.
In order to get some better approximation, one may take into account
either some parts or even the whole tail sum as well.
Adding the exact simulation of $\trunc{R_1}(h)$ in \eqref{eq:a0-rest},
which belongs to the Fourier coefficient $a_0^i$, results in the second
algorithm presented in
Section~\ref{subsec:alg-milstein} which is due to Milstein~\cite{MR1335454}.
In contrast to Milstein, in the seminal paper by Wiktorsson~\cite{MR1843055}
the whole tail sum $\trunc{R_1}(h) + \trunc{R_2}(h)$ is approximated.
The algorithm proposed by Wiktorsson is described in
Section~\ref{subsec:alg-wiktorsson} and has some improved
order of convergence compared to the first two algorithms. Finally, the recently
proposed algorithm by Mrongowius and R\"o{\ss}ler~\cite{Mrongowius2021}
combines the ideas of the algorithms due to Milstein and Wiktorsson by
simulating $\trunc{R_1}(h)$ exactly and by approximating the tail sum
$\trunc{R_2}(h)$ in a similar way as in the algorithm by Wiktorsson.
The algorithm by Mrongowius and R\"o{\ss}ler is described in
Section~\ref{subsec:alg-mronroe}.
In order to efficiently implement the two algorithms proposed by Wiktorsson
and by Mrongowius and R\"o{\ss}ler,
the central idea is to replace the involved matrices by the actions they encode.
This also helps avoiding computationally intensive Kronecker products.
In the following we discuss in detail mathematically equivalent
reformulations of these algorithms that finally lead to efficient
implementations with significantly reduced computational complexity.

Since the symmetric part of the iterated integrals is known explicitly, all
algorithms under consideration can be interpreted as algorithms for the
simulation of the skew-symmetric part which is exactly the L\'{e}vy area
$A(h)$. Moreover, due to the scaling
relationship~\eqref{eq:iteratedintegrals-time-scaling}
it is sufficient to focus our attention on the simulation of the L\'{e}vy area
with step size $h=1$ because the general case can be obtained by rescaling,
i.e., by multiplying the simulated L\'{e}vy area with the desired step size.
In the following
discussions of the different algorithms, we always assume that the
increments $W_h^i$ for $1 \leq i \leq m$, which are independent and
$\Normal{0}{h}$ distributed Gaussian random variables, are given.

\subsection{The Fourier algorithm}
\label{subsec:alg-fourier-series}
What we call here the Fourier algorithm can be considered the `easy'
approach --- not taking into account any approximation of the rest terms.
This variant has not been getting much attention in the literature, since with
very little additional
effort the error can be reduced by a factor of $1/\sqrt{3}$, see, e.g.,
\cite{Foster2022} and
Section~\ref{sec:error-estimates}.
However, since this is the foundation for all considered algorithms every
implementation detail automatically benefits all algorithms. Thus we discuss
it as its own algorithm.
\subsubsection{Derivation of the Fourier algorithm}
\label{subsubsec:fourier-derivation}
Let $\Se \in \mathbb{N}$ be given. Then, the approximation using the
Fourier algorithm is defined as
\begin{equation} \label{eq:fourier-levy-approx}
    \LAFS(h) \coloneqq \trunc{A}(h) ,
\end{equation}
i.e., just the truncated L\'{e}vy area without any further rest approximations.
Defining
\begin{equation} \label{eq:Def-alpha-beta-N01}
    \alpha_r^i = \sqrt{\frac{2 \pi^2 r^2}{h}} a_r^i
    \quad \text{ and } \quad
    \beta_r^i = \sqrt{\frac{2 \pi^2 r^2}{h}} b_r^i
\end{equation}
for $1 \leq i \leq m$ and $1 \leq r \leq \Se$ it follows that these random variables
are independent and standard Gaussian distributed. The approximation in
terms of these standard Gaussian random variables results in
\begin{equation} \label{eq:fourier-levy-approx-std-gaussian}
    \LAFS(h) = \frac{h}{2 \pi} \sum_{r=1}^{\Se} \frac{1}{r} \,
    \bigg( \alpha_r \bigg( \beta_r-\sqrt{\frac{2}{h}} W_h \bigg)^\top
    - \bigg( \beta_r-\sqrt{\frac{2}{h}} W_h \bigg) \alpha_r^\top \bigg),
\end{equation}
where $\alpha_r = (\alpha_r^i)_{1\leq i\leq m}$ and $\beta_r =
(\beta_r^i)_{1\leq i\leq m}$ are the corresponding random vectors.
\subsubsection{Implementation of the Fourier algorithm}
\label{subsubsec:fourier-implementation}
Looking at \eqref{eq:fourier-levy-approx-std-gaussian}, it is easy to see that
the dyadic products in the series are needed twice due to the relationship
$ \big( \beta_r-\sqrt{\tfrac{2}{h}} W_h \big) \alpha_r^\top = \big( \alpha_r
\big( \beta_r-\sqrt{\frac{2}{h}} W_h \big)^\top \big)^\top $. This means that
any efficient implementation of the above approximation should only compute
them once. Next, notice that this can be applied to the whole sum: instead of
evaluating each term in the series and then adding them together it is
more efficient to split the sum and evaluate it only once. Introducing
\begin{equation} \label{eq:S-FS-p}
    \SFS =  \sum_{r=1}^{\Se} \frac{1}{r} \, \Big( \alpha_r \Big(
    \beta_r - \sqrt{\tfrac{2}{h}} W_h \Big)^\top \Big)
\end{equation}
we can rewrite the Fourier approximation as
\begin{equation} \label{eq:std-FS-Alg-effective}
    \LAFS(h) = \frac{h}{2\pi} \Big( \SFS - {S^{\FS,\pp}}^{\top} \Big) .
\end{equation}
By first calculating and storing $\SFS$ and then computing $\LAFS(h)$
one can save $(p-1) m^2$ basic arithmetic operations
if \eqref{eq:std-FS-Alg-effective} is applied instead of naively
implementing \eqref{eq:fourier-levy-approx-std-gaussian}. The
algorithm based on \eqref{eq:std-FS-Alg-effective} is given as
pseudocode by Algorithm~\ref{alg:fourier}.

Next, we have a closer look at the calculation of $\SFS$.
For an implementation that results in a fast computation of the iterated
stochastic integrals, we want to exploit fast matrix multiplication routines
as provided by, e.g., specialised ``BLAS''\footnote{\url{http://www.netlib.org/blas/}}
(Basic Linear Algebra Subprograms) implementations like
OpenBLAS\footnote{\url{https://www.openblas.net/}} or the
Intel\textsuperscript{\textregistered} Math Kernel
Library\footnote{\url{https://software.intel.com/en-us/mkl}}.
In fact, if we gather the coefficient column vectors $\alpha_r$ and
$\beta_r$ for $1\leq r \leq \Se$ into the two $m \times \Se$-matrices
\begin{equation*}
    \alpha = (\alpha_1 \mid \ldots \mid \alpha_\Se)
    \quad \text{ and } \quad
    \tilde{\beta} = \left( \left. \frac{\beta_1 - \sqrt{\tfrac{2}{h}} W_h}{1}
    \right| %
    \ldots \left| %
    \frac{\beta_\Se - \sqrt{\tfrac{2}{h}} W_h}{\Se} \right. \right)
\end{equation*}
we can compute $\SFS$ as a matrix product $\SFS = \alpha
\tilde{\beta}^\top$ utilizing the aforementioned fast multiplication algorithms.
Of course, the drawback is that we have to store all coefficients in memory
at the same time whereas before we only needed to store $\alpha_r$ and
$\beta_r$ of the current iteration for the summation.

There is also a middle
ground between these two extremes --- for some $n \in \{1, \ldots, p\}$
generate the $m \times n$-matrices $\alpha^{(k)}$ and $\beta^{(k)}$ for
$1 \leq k \leq \lceil \tfrac{\Se}{n} \rceil$ by partitioning the columns of
the matrices $\alpha$ and $\tilde{\beta}$
in the same way into $\lceil \frac{\Se}{n} \rceil$ groups $\alpha^{(k)}$
and $\tilde{\beta}^{(k)}$ for $1 \leq k \leq \lceil \frac{\Se}{n} \rceil$,
each of maximal $n$ vectors, and adding together sequentially to get
$\SFS = \sum_{k=1}^{\lceil {\Se}/{n} \rceil} \alpha^{(k)}
\tilde{\beta}^{(k) \top}$.
Here, the firstly discussed extremes correspond to $n=1$ and $n=\Se$,
respectively.
\begin{algorithm}[htbp]
	\normalsize
	\caption{L\'{e}vy area --- Fourier method}
	\label{alg:fourier}
	\begin{algorithmic}[1]
		\Require{W \Comment{m-dimensional Wiener increment
				for $h=1$ [vector of length m]}}
		\Require{p \Comment{truncation parameter for tail sum [natural number]}}
		\Statex

		\Function{levyarea\textunderscore fourier}{$ W, p $}
		\State $ \alpha \gets  $ \Call{randn}{$ m, p $}  \Comment{generate
			$m \times p$ Gaussian random numbers}
		\State $ \beta \gets  $ \Call{randn}{$ m, p $}  \Comment{generate
			$m \times p$ Gaussian random numbers}

		\ForAll{columns $ \beta_r $ of $ \beta $}
		\State $\beta_r \gets \frac{1}{r} \big( \beta_r - \sqrt{2} W \big)$
		\EndFor
		\State $S \gets \alpha \beta^\top$ \Comment{call BLAS gemm function
			here}
		\State $A \gets \frac{1}{2\pi} \big(S-S^{\top} \big) $
		\State \Return $A$ \Comment{return approximation of L\'{e}vy area
			$A(1)$}
		\EndFunction
	\end{algorithmic}
\end{algorithm}
\subsection{The Milstein algorithm}
\label{subsec:alg-milstein}
Now we proceed with the algorithm that Milstein first investigated in his
influential work \cite{MR1335454}. He called it the Fourier method, but he
already incorporated a simple rest approximation compared to what we
call the Fourier algorithm in Section~\ref{subsec:alg-fourier-series}.
This algorithm has been generalised to multiple iterated integrals
by Kloeden, Platen and Wright  \cite{MR1178485}, see also \cite{MR1214374}.
Note that Dickinson~\cite{MR2352954} showed that this method is
asymptotically optimal in the setting where only a finite number of linear
functionals of the Wiener process are allowed to be used.
\subsubsection{Derivation of the Milstein algorithm}
\label{subsubsec:milstein-derivation}
For some prescribed $\Se \in \mathbb{N}$, the rest term approximation
employed by Milstein~\cite{MR1335454} is
concerned with the exact simulation of the term $\trunc{R_1}(h)$
given by \eqref{eq:a0-rest}. Utilizing the independence of the coefficients
$a_r^i$ and their distributional properties \eqref{eq:fourier-coeffs-distribution},
the random vector $ \sum_{r=\Se+1}^{\infty} a_r $ possesses
a multivariate Gaussian distribution with zero mean and covariance matrix
$ \frac{h}{2\pi^2} \sum_{r=\Se+1}^{\infty} \frac{1}{r^2}  I_m =
\frac{h}{2\pi^2}\psi_1(\Se+1) I_m $.
Here, $\psi_1$ denotes the trigamma function defined by $\psi_1(z) =
\frac{\mathrm{d}^2}{\mathrm{d}z^2} \ln(\Gamma(z))$ with respect to the
gamma function $\Gamma$ which can also be represented as the series
$\psi_1(z) = \sum_{n=0}^\infty \frac{1}{(z+n)^2}$. Thus, the vector $\gamma_1$
defined by
\begin{equation}\label{eq:milstein-mu}
    \gamma_1 = \sqrt{ \frac{2 \pi^2}{h \, \psi_1(\Se+1)} }
    \sum_{r=\Se+1}^{\infty} a_r
\end{equation}
is $\Normal{0_m}{I_m}$ distributed. Note that $\gamma_1$ is
independent from $W_h$,  from all $ \beta_k $ as well as from $\alpha_r$
for $r=1,\ldots,p$. Now, one can rewrite the rest term as
\begin{equation} \label{eq:milstein-rest-approx}
	\trunc{R_1}(h) = \sqrt{\frac{h \, \psi_1(\Se+1)}{2\pi^2}}
    \big( W_h\gamma_1^\top - \gamma_1 W_h^\top \big)
\end{equation}
and the approximation of the L\'{e}vy area $\LevyArea(h) \approx \LAMil(h)$
with the Milstein algorithm is defined by
\begin{equation} \label{eq:milstein-levy-approx}
    \LAMil(h) = \LAFS(h) + \trunc{R_1}(h)
\end{equation}
with $\trunc{R_1}(h)$
given by \eqref{eq:milstein-rest-approx}.
We emphasize that the exact simulation of $\trunc{R_1}(h)$ is
possible because the distribution of $\gamma_1$ is explicitly known
to be multivariate standard Gaussian.
Provided $W_h$ is given, we have to simulate \eqref{eq:std-FS-Alg-effective}
and we additionally have to add the term $\trunc{R_1}(h)$
given in
\eqref{eq:milstein-rest-approx} for the Milstein algorithm in order to reduce
the error. Note that the exact simulation of $\trunc{R_1}(h)$ by
\eqref{eq:milstein-rest-approx} does not improve the order of convergence.
However, it provides a by a factor of $1/\sqrt{3}$ smaller error constant and thus,
in most cases, less random variables are needed compared to the Fourier
algorithm in order to guarantee some prescribed error bound.
\subsubsection{Implementation of the Milstein algorithm}
\label{subsubsec:milstein-implementation}
For the implementation of the Milstein algorithm we again separate
the L\'{e}vy area approximation into two appropriate parts $\SMil$
and $-S^{\Mil,\pp\top}$.
Then, the Milstein approximation is computed as
\begin{equation}
    \LAMil(h) = \frac{h}{2\pi} \Big( \SMil-{S^{\Mil,\pp}}^\top \Big)
\end{equation}
where looking at \eqref{eq:milstein-rest-approx} and taking into account
\eqref{eq:S-FS-p}, we see that $\SMil$ can be expressed as
\begin{equation}
    \SMil = \SFS + \sqrt{2\psi_1(\Se+1)} \, \frac{W_h}{\sqrt{h}} \, \gamma_1^\top
\end{equation}
and thus can be easily simulated. Finally, the Milstein algorithm is presented
as pseudocode in Algorithm~\ref{alg:milstein}.
\begin{algorithm}[htb]
    \normalsize
    \caption{L\'{e}vy area --- Milstein method}
    \label{alg:milstein}
    \begin{algorithmic}[1]
        \Require{W \Comment{m-dimensional Wiener increment
        	for $h=1$ [vector of length m]}}
        \Require{p \Comment{truncation parameter for tail sum [natural number]}}
        \Statex

        \Function{levyarea\textunderscore milstein}{$ W, p $}
        \State $ \alpha \gets  $ \Call{randn}{$ m, p $}  \Comment{generate
            $m \times p$ Gaussian random numbers}
        \State $ \beta \gets  $ \Call{randn}{$ m, p $}  \Comment{generate
            $m \times p$ Gaussian random numbers}
        \State $ \gamma_1 \gets  $ \Call{randn}{$ m $}  \Comment{generate
            $m \times 1$ Gaussian random numbers}

        \ForAll{columns $ \beta_r $ of $ \beta $}
        \State $\beta_r \gets \frac{1}{r} \big( \beta_r - \sqrt{2} W \big)$
        \EndFor
        \State $S \gets \alpha\beta^\top $ \Comment{call BLAS gemm function here}
        \State $S \gets S + \sqrt{2 \, \Call{trigamma}{p+1}} \, W \gamma_1^\top$ \Comment{simple tail approximation}
        \State $A \gets \frac{1}{2\pi} \Big( S-S^{\top} \Big) $
        \State \Return $A$ \Comment{return approximation of L\'{e}vy area
            $A(1)$}
        \EndFunction
    \end{algorithmic}
\end{algorithm}
\subsection{The Wiktorsson algorithm}
\label{subsec:alg-wiktorsson}
In the Milstein algorithm, the tail sum $ \trunc{R}(h) $ is only partially
taken into account which reduces the error constant but has no influence on
the order of convergence w.r.t.\ the number of necessary realizations of
independent Gaussian random variables.
An improvement of the order of convergence has been first achieved by
Wiktorsson~\cite{MR1843055}. He did not incorporate $\trunc{R_1}(h)$
on its own but instead tried to find an approximation of the whole tail sum
$\trunc{R}(h)$ by analysing its conditional distribution in a new way.
In this section, we briefly explain the main idea Wiktorsson used to derive his
sophisticated approximation of the truncation terms. Therefore, we first
introduce a suitable notation together with some useful identities.
Wiktorsson's idea is also the basis for the algorithm proposed by Mrongowius
and R\"o{\ss}ler~\cite{Mrongowius2021}.
\subsubsection{Vectorization and the Kronecker product representation}
\label{subsubsec:vec-identities}
In the following, let $\vecop \colon \mathbb{R}^{m \times m} \to
\mathbb{R}^{m^2}$ be the column-wise vectorization operator for quadratic
matrices and let $\matop \colon \mathbb{R}^{m^2} \to
\mathbb{R}^{m \times m}$ be its inverse mapping such that
$\matop(\vecop(A)) = A$ for all matrices $A \in \mathbb{R}^{m \times m}$.
Then, for the Kronecker product of two vectors $u,v \in \mathbb{R}^m$
it holds $v \otimes u = \vecop(u v^\top)$.
Further, let
$P_m \in \mathbb{R}^{m^2 \times m^2}$ be the permutation matrix that satisfies
$P_m(u \otimes v) = v \otimes u$ for all $u,v \in \mathbb{R}^m$.
Note that it can be explicitly constructed as $ P_m = \sum_{i,j = 1}^m
e_ie_j^\top \otimes e_je_i^\top $ where $e_i$ denotes the $i$th canonical
basis vector in $\mathbb{R}^m$. Then, it follows that
\begin{equation} \label{eq:Pmvec}
    P_m \vecop(A) = \vecop(A^\top)
\end{equation}
and we have the identity
\begin{equation} \label{eq:kronvec}
    (A \otimes I_m) \vecop(B) = \vecop(B A^\top)
\end{equation}
for $A, B \in \mathbb{R}^{m \times m}$, see, e.g., \cite{MR611262,MR520247}.

Since we are dealing with skew-symmetric matrices, it often suffices to consider
the $M = \tfrac{m(m-1)}{2}$ entries below the diagonal. Therefore,
Wiktorsson~\cite{MR1843055} applies the matrix $K_m \in \mathbb{R}^{M \times
m^2}$ that selects the entries below the diagonal of a vectorized $m \times
m$-matrix. This matrix can be explicitly defined by
\begin{equation*}
	K_m =
	\sum_{1 \leq j<i \leq m} \tilde{e}_{i-\frac{j(j+1)}{2}+(j-1)m}
	(e_j^\top \otimes e_i^\top)
\end{equation*}
where $ \tilde{e}_k $ is the $ k $th canonical basis vector in $ \mathbb{R}^{M} $.
Note that the matrix $ K_m $ is clearly not invertible when considered as
acting from $ \mathbb{R}^{m^2} $ to $ \mathbb{R}^M $, but it has $K_m^\top$
as a right inverse. If we restrict the domain of $ K_m $ to the $M$-dimensional subspace
$\{ \vecop(A) : A \text{ strictly lower triangular} \} \subset\mathbb{R}^{m^2}$, then $ K_m^\top $ is
also a left inverse to $ K_m $. That means, we can reconstruct at least all
skew-symmetric matrices in the sense that
\begin{equation}\label{eq:reconstruct-skew-vector}
    (I_{m^2} -P_m) K_m^\top K_m \vecop(A) = \vecop(A)
\end{equation}
for any skew-symmetric $A \in \mathbb{R}^{m \times m}$.
As a result of this, it holds that
\begin{equation} \label{eq:IPKKIP}
	(I_{m^2}-P_m) K_m^\top K_m (I_{m^2}-P_m) = I_{m^2}-P_m \,.
\end{equation}
On the other hand, if we are given a vector $ a\in\mathbb{R}^M $ we get the
corresponding skew-symmetric matrix $ A\in\mathbb{R}^{m\times m} $
using the $ \matop $-operator as
\begin{equation}\label{eq:reconstruct-skew-matrix}
    A = \matop \big( (I_{m^2} -P_m) K_m^\top a \big) \,.
\end{equation}
\subsubsection{Derivation of the Wiktorsson algorithm}
\label{subsubsec:wiktorsson-derivation}
Due to the skew-symmetry of the L\'{e}vy area matrix $A(h)$, we can
restrict our considerations to the vectorization of the lower triangle.
In this way, following the idea of Wiktorsson~\cite{MR1843055}
it is possible to analyse the conditional covariance structure
of the resulting vector and to reduce the L\'{e}vy area approximation problem
to the problem of finding a good approximation of the covariance matrix.
Applying this approach to the tail sum $\trunc{R}(h)$ of the truncated series
\eqref{eq:levy-area-trunc} and using \eqref{eq:Def-alpha-beta-N01}, it holds
that
\begin{equation}\label{eq:vec-R-p}
    \vecop(\trunc{R}(h)) = \frac{h}{2 \pi} \sum_{r=\Se+1}^{\infty}
    (P_m - I_{m^2}) \bigg( \alpha_r \otimes \frac{1}{r}
    \bigg(\beta_r-\sqrt{\frac{2}{h}}W_h \bigg) \bigg) ,
\end{equation}
which is an $\mathbb{R}^{m^2}$-valued random vector. Since $\trunc{R}(h)$
is skew-symmetric as well, it suffices to consider the $M$ entries below
the diagonal. Let $\trunc{r}(h) = K_m \vecop(\trunc{R}(h))$ denote the
$M$-dimensional random vector of the elements below the diagonal of
$\trunc{R}(h)$.
Now, for each $\Se$ the random vector $\frac{\sqrt{2} \pi}{h \sqrt{\psi_1(\Se+1)}}
\trunc{r}(h)$ is conditionally Gaussian with conditional expectation $0_M$
and some conditional covariance matrix $\trunc{\Sigma}=\trunc{\Sigma}(h)$ which
depends on the random vectors $\beta_r$ , $r \geq \Se+1$ and $ W_h $.
We use a slightly different scaling than Wiktorsson, because we think it better highlights the structure of the covariance matrices. In this setting the covariance matrix is given by
\begin{align}
\begin{split}\label{eq:sigma-p-wik}
    \trunc{\Sigma} &= \frac{1}{2}K_m(I_{m^2}-P_m)
    \cdot\Big(\frac{1}{\psi_1(p+1)}\sum\limits_{r=p+1}^\infty\frac{1}{r^2}\big(\beta_r-\sqrt{\frac{2}{h}}W_h \big)\big(\beta_r-\sqrt{\frac{2}{h}}W_h \big)^\top\otimes I_m\Big) \\
    &\quad\cdot (I_{m^2}-P_m)K_m^\top \,.
\end{split}
\end{align}
Therefore, there exists a standard normally distributed $M$-dimensional
random vector $\gamma$ such that
\begin{align} \label{eq:gamma-p-R-p}
	\sqrt{\trunc{\Sigma}} \, {\gamma}
	= \frac{\sqrt{2} \pi}{h \sqrt{\psi_1(p+1)}} \, \trunc{r}(h) \quad \Prob\text{-a.s.}
\end{align}
As $\Se \to \infty$, these covariance matrices converge in the
\LF-norm to the matrix
\begin{equation} \label{eq:sigmainf}
    \Sigma^{\infty} = I_M + \frac{1}{h} K_m (I_{m^2} -P_m)
    (W_h W_h^{\top} \otimes I_m ) (I_{m^2} -P_m) K_m^{\top}
\end{equation}
that is independent of the random vectors $\beta_r$, see
\cite[Thm.~4.2]{MR1843055}.
Now, we can take the underlying $\Normal{0_M}{I_M}$ distributed random vector
$\gamma$ from \eqref{eq:gamma-p-R-p} in order to approximate the
tail sum as
\begin{equation}
    \trunc{r}(h) \approx \frac{h\sqrt{\psi_1(\Se+1)}}{\sqrt{2}\pi} \,
    \sqrt{\Sigma^{\infty}} \, {\gamma}
\end{equation}
where the exact conditional covariance matrix $\trunc{\Sigma}$ in
\eqref{eq:gamma-p-R-p} is replaced by $\Sigma^{\infty}$, which can be
computed explicitly. Finally, we can approximate the corresponding
skew-symmetric matrix $\trunc{R}(h)$ by the reconstructed matrix
$\app{R}^{\Wik,\pp}(h)$ defined as
\begin{equation}\label{eq:wiktorsson-rest-approx}
    \app{R}^{\Wik,\pp}(h)
    = \matop \Big( (I_{m^2} -P_m) K_m^\top \frac{h}{\sqrt{2} \pi} \sqrt{\psi_1(p+1)}
    \sqrt{\Sigma^{\infty}} {\gamma} \Big)
\end{equation}
due to \eqref{eq:reconstruct-skew-matrix}. The matrix square root of
$\Sigma^{\infty}$ in the above formula can be explicitly calculated as
\begin{equation}\label{eq:sqrtsigmainf}
    \sqrt{\Sigma^{\infty}} = \frac{ \Sigma^{\infty}
    	+ \sqrt{1 +\frac{\norm{W_h}^2}{h}}
    	I_M }{ 1 +\sqrt{1 +\frac{\norm{W_h}^2}{h}} } ,
\end{equation}
see \cite[Thm.~4.1]{MR1843055}. Thus, the approximation of the L\'{e}vy area
$\LevyArea(h) \approx \LAWik(h)$ using the tail sum approximation as
proposed by Wiktorsson is defined as
\begin{equation}\label{eq:wiktorsson-levy-approx}
    \LAWik(h) = \LAFS(h) + \app{R}^{\Wik,\pp}(h)
\end{equation}
with $\LAFS(h)$ and $\app{R}^{\Wik,(p)}(h)$ given by
\eqref{eq:fourier-levy-approx} and \eqref{eq:wiktorsson-rest-approx},
respectively.
\subsubsection{Implementation  of the Wiktorsson algorithm}
\label{subsubsec:wiktorsson-implementation}
Having a closer look at the tail approximation \eqref{eq:wiktorsson-rest-approx},
it is easy to see that it is not directly suitable for an efficient implementation.
Multiple matrix-vector products involving matrices of
size $m^2 \times m^2$ result in an algorithmic complexity of roughly
$\Oo(m^4)$. Additionally the Kronecker product with an identity matrix as
it appears in the representation of $\Sigma^{\infty}$, from an algorithmic
point of view, only serves to duplicate values. Among other things, this leads
to unnecessary high memory requirements. Therefore, we derive a better
representation using the properties of the $\vecop$-operator mentioned in
Section~\ref{subsubsec:vec-identities}.

As a first step, let us insert expression \eqref{eq:sigmainf} for $\Sigma^\infty$
into the representation \eqref{eq:sqrtsigmainf} of $\sqrt{\Sigma^\infty}$.
Then, we get
\begin{align}
    \sqrt{\Sigma^\infty} &= \frac{ I_M + \frac{1}{h} K_m (I_{m^2} -P_m)
    (W_h W_h^\top \otimes I_m) (I_{m^2} -P_m) K_m^\top + \sqrt{1
    +\frac{\| W_h \|^2}{h}} I_M }{ 1 +\sqrt{1 +\frac{\| W_h \|^2}{h}} }
	\nonumber\\
    &= \frac{ K_m (I_{m^2} -P_m) (W_h W_h^\top \otimes I_m )
    (I_{m^2} -P_m) K_m^\top }{ h \Big(1 +\sqrt{1 +\frac{\| W_h \|^2}{h}} \Big)}
	+ I_M . \label{eq:sqrtsigmainf2}
\end{align}
Using \eqref{eq:IPKKIP} and \eqref{eq:wiktorsson-rest-approx}, we obtain
that
\begin{equation} \label{eq:Rpapprox2}
	\begin{split}
    \app{R}^{\Wik,\pp}(h) &= \frac{h}{2\pi} \matop \bigg( (I_{m^2} -P_m)
    \bigg( \frac{ (W_h W_h^\top \otimes I_m) (I_{m^2} -P_m)
    \sqrt{2 \psi_1(\Se+1)} K_m^\top {\gamma} }{ h \Big( 1 +\sqrt{1 +\frac{ \| W_h \|^2
     }{h}} \Big)} \\
    &\quad + \sqrt{2 \psi_1(\Se+1)} K_m^\top {\gamma} \bigg)\bigg) .
    \end{split}
\end{equation}
Next, observing that ${\gamma}$ is always immediately reshaped by $K_m^\top$
we define the random $m \times m$-matrix $\varGamma = \matop \big( K_m^\top
\underline{\underline{\underline{}}}{\gamma} \big)$ to be the reshaped matrix.
As a result of this, it holds
$ K_m^\top {\gamma} = \vecop({\varGamma})$ with ${\varGamma}_{i,j} \sim
\Normal{0}{1}$ for $i > j$ and ${\varGamma}_{i,j}=0$ for $i\leq j$.
Now, we apply \eqref{eq:kronvec} as well as \eqref{eq:Pmvec} in order to arrive
at the final representation
\begin{equation}
    \begin{split}
    \app{R}^{\Wik,\pp}(h) &= \frac{h}{2\pi} \matop \bigg( (I_{m^2} -P_m) \vecop
    \bigg( \frac{\sqrt{2 \psi_1(\Se+1)}({\varGamma} - \varGamma^{\top})
    W_h W_h^\top }{ h \Big( 1 +\sqrt{1 +\frac{\| W_h \|^2}{h}} \Big)} \\
	&\quad + \sqrt{2 \psi_1(\Se+1)}{\varGamma} \bigg) \bigg) .
    \end{split}
\end{equation}
With \eqref{eq:S-FS-p} we define
\begin{equation}
    \SWik = \SFS
    + \frac{ \sqrt{2 \psi_1(\Se+1)} ({\varGamma} - \varGamma^{\top}) }{
    1 +\sqrt{1 +\frac{\| W_h \|^2}{h}} } \cdot \frac{W_h W_h^\top}{h}
    + \sqrt{2 \psi_1(\Se+1)}{\varGamma}
\end{equation}
which allows for an efficient implementation of Wiktorsson's algorithm using
\begin{equation}\
    \LAWik(h) = \frac{h}{2\pi} \Big( \SWik - {S^{\Wik,\pp}}^\top \Big) .
\end{equation}
This version of Wiktorsson's algorithm that is optimized for an efficient
implementation can be found in Algorithm~\ref{alg:wiktorsson} presented as
pseudocode.
\begin{algorithm}[htb]
    \normalsize
    \caption{L\'{e}vy area --- Wiktorsson method}
    \label{alg:wiktorsson}
    \begin{algorithmic}[1]
        \Require{W \Comment{m-dimensional Wiener increment for $h=1$
        		[vector of length m]}}
        \Require{p \Comment{truncation parameter for tail sum [natural number]}}
        \Statex

        \Function{levyarea\textunderscore wiktorsson}{$ W, p $}
        \State $ \alpha \gets  $ \Call{randn}{$ m, p $}  \Comment{generate
        	$m \times p$ Gaussian random numbers}
        \State $ \beta \gets  $ \Call{randn}{$m, p$}  \Comment{generate
        	$m \times p$ Gaussian random numbers}
		\State $ \varGamma \gets  $ \Call{randn\textunderscore tril}{$ m, m $}
			\Comment{gen.\ strictly lower triangular matrix of $ \frac{m(m-1)}{2} $
				Gaussian r.n.}

        \ForAll{columns $ \beta_r $ of $ \beta $}
        \State $ \beta_r \gets \frac{1}{r} \big( \beta_r - \sqrt{2}W \big) $
        \EndFor
        \State $ S \gets \alpha \beta^\top $ \Comment{call BLAS gemm function
        	here}
        \State $ {\varGamma} \gets \sqrt{2\,\Call{trigamma}{p+1}} \, \varGamma$
        \State $ S \gets S + \big( 1 +\sqrt{1 +\| W \|^2} \big)^{-1}
        \big( {\varGamma} -{\varGamma}^{\top} \big) W W^\top
        + {\varGamma} $
        \State $ A \gets \frac{1}{2\pi} \big( S-S^{\top} \big) $
        \State \Return $ A $ \Comment{return approximation of L\'{e}vy area $A(1)$}
        \EndFunction
    \end{algorithmic}
\end{algorithm}
\subsection{The Mrongowius--Rößler algorithm}
\label{subsec:alg-mronroe}
Compared to the Fourier algorithms that make use of linear functionals of the
Wiener processes only, the algorithm proposed by Wiktorsson \cite{MR1843055}
reduces the computational cost significantly by including also a kind of non-linear
information of the Wiener processes for the additional approximation of the
remainder terms. However, Wiktorsson's algorithm needs the calculation of
the square root of a $M \times M$ covariance matrix, which can be expensive
for high dimensional problems.
The following algorithm that has been recently proposed by Mrongowius and
R\"o{\ss}ler~\cite{Mrongowius2021} improves the algorithm
by Wiktorsson
such that the error estimate allows for a smaller constant by a factor
$1/{\sqrt{5}}$, see Section~\ref{sec:error-estimates}. In addition,
the covariance matrix used in the algorithm by Mrongowius and R\"o{\ss}ler
is just the identity matrix and thus independent of $W_h$.
Compared to Wiktorsson's algorithm, this makes the calculation of the
square root of the full covariance matrix $\Sigma^\infty$ depending on $W_h$
in \eqref{eq:sqrtsigmainf2} obsolete. This allows for saving
computational cost, especially in the context of the simulation of SDE and
SPDE solutions where the value of $W_h$ varies each time step.
Thus, the algorithm by Mrongowius and R\"o{\ss}ler
saves computational cost and allows for an easier implementation that is
discussed in the following.
\subsubsection{Derivation of the Mrongowius--R\"o{\ss}ler algorithm}
\label{subsubsec:mronroe-derivation}
Instead of approximating the whole rest term $\trunc{R}(h)$ at once as in
Wiktorsson's algorithm, Mrongowius and R\"o{\ss}ler combine the exact
approximation of the rest term $ \trunc{R_1}(h) $ in \eqref{eq:milstein-rest-approx}
with an approximation of the second rest term $ \trunc{R_2}(h) $ that is
related to Wiktorsson's approach.
For the approximation of $\trunc{R_2}(h)$, we rewrite this term as
\begin{equation}\label{eq:vec-R-2-p}
	\vecop(\trunc{R_2}(h)) = \frac{h}{2 \pi} \sum_{r=p+1}^{\infty} \frac{1}{r}
	(P_m -I_{m^2}) ( \alpha_r \otimes \beta_r )
\end{equation}
and we define the $M$-dimensional vector $\trunc{r_2}(h) = K_m
\vecop(\trunc{R_2}(h))$. Then, the random vector
$\frac{\sqrt{2} \pi}{h \sqrt{\psi_1(\Se+1)}} \trunc{r_2}(h)$ is conditionally
Gaussian with conditional expectation $0_M$ and conditional covariance
matrix $\trunc{\Sigma_2}$ that depends on $\alpha_r$ , $r \geq \Se+1$.
Here, the covariance matrix is given by (cf.\ \eqref{eq:sigma-p-wik})
\begin{align}\label{eq:sigma-p-mr}
    \trunc{\Sigma_2} = \frac{1}{2}K_m(I_{m^2}-P_m) \Big(I_m\otimes\frac{1}{\psi_1(p+1)}\sum\limits_{r=p+1}^\infty\frac{1}{r^2}\alpha_r\alpha_r^\top\Big) (I_{m^2}-P_m)K_m^\top \,.
\end{align}
As a result of this, it holds
\begin{align*}
	\frac{\sqrt{2} \pi}{h \sqrt{\psi_1(\Se+1)}} \trunc{r_2}(h)
	= \sqrt{\trunc{\Sigma_2}} {\gamma_2}
\end{align*}
with some $\Normal{0_M}{I_M}$ distributed random vector ${\gamma_2}$ that
is given by
\begin{align} \label{eq:RV-Psi-2-n-Exact}
	{\gamma_2} = \frac{\sqrt{2} \pi}{h \sqrt{\psi_1(\Se+1)}}
	\Big( \trunc{\Sigma_2} \Big)^{-1/2} \trunc{r_2}(h) \, .
\end{align}
Mrongowius and R\"o{\ss}ler showed that these covariance matrices converge as
$\Se \to \infty$ in the \LF-norm to the very simple,
constant identity matrix \cite[Prop. 4.2]{Mrongowius2021}
\begin{equation}
	\Sigma_2^{\infty}
    = I_M \,.
\end{equation}
Next, the approximation of the remainder term $\trunc{r_2}(h)$ based on the
approximate covariance matrix $\Sigma_2^{\infty}$ is calculated as
\begin{align} \label{Def-hat-A2n}
	\trunc{\app{r_2}}(h) &= \frac{h}{\sqrt{2} \pi} \sqrt{\psi_1(\Se+1)}
	\big( \Sigma_2^{\infty} \big)^{1/2} {\gamma_2}
	= \frac{h}{\sqrt{2} \pi} \sqrt{\psi_1(\Se+1)} {\gamma_2}
\end{align}
where ${\gamma_2}$ is the $\Normal{0_M}{I_M}$ distributed random vector in
\eqref{eq:RV-Psi-2-n-Exact}.
From this vector we can rebuild the full $ m\times m $ matrix as shown in
\eqref{eq:reconstruct-skew-matrix} by
\begin{equation}\label{eq:mronroe-rest-approx}
    \app{R_2}^{\MR,\pp}(h) =
    \matop \Big( (I_{m^2}-P_m) K_m^\top \trunc{\app{r_2}}(h) \Big) \, .
\end{equation}
Now, the approximation of the L\'{e}vy area by the Mrongowius--Rößler
algorithm is defined by the standard Fourier approximation combined with
two rest term approximations
\begin{equation}\label{eq:mronroe-levy-approx}
    \LAMR(h) =
    \LAFS(h) + \trunc{R_1}(h) + \app{R_2}^{\MR,\pp}(h)
\end{equation}
with
$\LAFS(h)$, $\trunc{R_1}(h)$ and $\app{R_2}^{\MR,\pp}(h)$
defined in \eqref{eq:fourier-levy-approx}, \eqref{eq:milstein-rest-approx} and
\eqref{eq:mronroe-rest-approx}, respectively.
\subsubsection{Implementation of the Mrongowius--R\"o{\ss}ler algorithm}
\label{subsubsec:mronroe-implementation}
Similar to Subsection \ref{subsubsec:wiktorsson-implementation}, we simplify
the expression \eqref{eq:mronroe-rest-approx} by eliminating the
matrices $ P_m $ and $ K_m $ as well as the costly reshaping operation.
This can be done in a much easier way due to the simple structure of the
covariance matrix for this algorithm.
First, we define $ {\varGamma_2} = \matop \big( K_m^\top {\gamma_2} \big) $
to be the lower triangular matrix corresponding to $ {\gamma_2} $. Then, it
follows that
\begin{equation}
	\app{R_2}^{\MR,\pp}(h) = \frac{h}{{2} \pi} \sqrt{2 \psi_1(\Se+1)}
	\big( {\varGamma_2} - \varGamma_2^{\top} \big) \, .
\end{equation}
Now, with $\SFS$ in \eqref{eq:S-FS-p} we can define
\begin{equation}
    \SMR = \SFS
    + \sqrt{2\psi_1(\Se+1)} \, \frac{W_h}{\sqrt{h}} \, \gamma_1^\top
    + \sqrt{2\psi_1(\Se+1)} \, \varGamma_2
\end{equation}
and then we finally arrive at
\begin{equation}
    \LAMR(h) = \frac{h}{2\pi} \Big( \SMR - {S^{\MR,\pp}}^\top \Big)
\end{equation}
that allows for an efficient implementation. This formulation of the algorithm
due to Mrongowius and R\"o{\ss}ler is also given in Algorithm~\ref{alg:mronroe}
as pseudocode.
\begin{algorithm}[H]
	\normalsize
	\caption{L\'{e}vy area --- Mrongowius--R\"o{\ss}ler method}
	\label{alg:mronroe}
	\begin{algorithmic}[1]
		\Require{W \Comment{m-dimensional Wiener increment for $h=1$
				[vector of length m]}}
		\Require{p \Comment{truncation parameter for tail sum [natural number]}}
		\Statex

		\Function{levyarea\textunderscore mrongowius\textunderscore roessler}{$ W, p $}
		\State $ \alpha \gets  $ \Call{randn}{$m, p$}  \Comment{generate
			$m \times p$ Gaussian random numbers}
		\State $ \beta \gets  $ \Call{randn}{$m, p$}  \Comment{generate
			$m \times p$ Gaussian random numbers}
		\State $ \gamma_1 \gets  $ \Call{randn}{$m, 1$}  \Comment{generate
			$m \times 1$ Gaussian random numbers}
		\State $ \varGamma_2 \gets  $ \Call{randn\textunderscore tril}{$ m, m $}
			\Comment{gen.\ strictly lower triangular matrix of $ \frac{m(m-1)}{2} $ Gaussian r.n.}

		\ForAll{columns $ \beta_r $ of $ \beta $}
		\State $ \beta_r \gets \frac{1}{r} \big( \beta_r - \sqrt{2} \, W \big) $
		\EndFor
		\State $ S \gets \alpha \beta^\top$ \Comment{call BLAS gemm function
			here}
		\State $ S \gets S + \sqrt{2\,\Call{trigamma}{p+1}} (W \gamma_1^\top + \varGamma_2)$
		\State $ A \gets \frac{1}{2\pi} \big( S-S^{\top} \big) $
		\State \Return $A$ \Comment{return approximation of L\'{e}vy area
			$A(1)$}
		\EndFunction
	\end{algorithmic}
\end{algorithm}
\section{Error estimates and simulation of iterated stochastic integrals}
\label{sec:error-estimates}
The considered algorithms differ significantly with respect to their accuracy
and computational cost. Therefore, we first compare the corresponding
approximation error estimates in the \maxL-norm.
For $h>0$ and $p \in \mathbb{N}$, let
\begin{align} \label{eq:def-hat-I-Alg-p}
	\hat{\mathcal{I}}^{\Alg,\pp}(h) = \frac{W_h W_h^\top - h I_m}{2}
	+ \hat{A}^{\Alg,\pp}(h)
\end{align}
where '$\Alg$' denotes the corresponding algorithm that is applied for the
approximation of the L\'{e}vy areas.
\begin{theorem} \label{thm:approx-error-estimates}
	Let $h>0$, $p \in \mathbb{N}$, let $\mathcal{I}$ denote the matrix
	of all iterated stochastic integrals as in \eqref{eq:Iter-Ito-Int-Levy}
	and let $\hat{\mathcal{I}}$ denote the corresponding approximations
	defined in \eqref{eq:def-hat-I-Alg-p}.
	\begin{enumerate}[(i)]
		\item \label{thm:fourier-error}
		For the Fourier algorithm with $\LAFS(h)$ as in
		\eqref{eq:fourier-levy-approx-std-gaussian}, it holds
		\begin{equation} \label{eq:fourier-error}
			\|\mathcal{I}(h)-\hat{\mathcal{I}}^{\FS,\pp}(h)\|_\maxL
			\leq \sqrt{\frac{3}{2\pi^2}}\cdot\frac{h}{\sqrt{p}} \, .
		\end{equation}
		\item \label{thm:milstein-error}
		For the Milstein algorithm with $\LAMil(h)$ as in
		\eqref{eq:milstein-levy-approx}, it holds
		\begin{equation} \label{eq:milstein-error}
			\|\mathcal{I}(h)-\hat{\mathcal{I}}^{\Mil,\pp}(h)\|_\maxL
			\leq \sqrt{\frac{1}{2\pi^2}}\cdot\frac{h}{\sqrt{p}} \, .
		\end{equation}
		\item \label{thm:wiktorsson-error}
		For the Wiktorsson algorithm with $\LAWik(h)$ as in
		\eqref{eq:wiktorsson-levy-approx}, it holds
		\begin{equation} \label{eq:wiktorsson-error}
			\|\mathcal{I}(h)-\hat{\mathcal{I}}^{\Wik,\pp}(h)\|_\maxL
			\leq \sqrt{\frac{5m}{12\pi^2}}\cdot\frac{h}{p} \, .
		\end{equation}
		\item \label{thm:mronroe-error}
		For the Mrongowius--R\"o{\ss}ler algorithm with $\LAMR(h)$ as in
		\eqref{eq:mronroe-levy-approx}, it holds
		\begin{equation} \label{eq:mronroe-error}
			\|\mathcal{I}(h)-\hat{\mathcal{I}}^{\MR,\pp}(h)\|_\maxL
			\leq \sqrt{\frac{m}{12\pi^2}}\cdot\frac{h}{p} \, .
		\end{equation}
	\end{enumerate}
\end{theorem}
To the best of our knowledge, there is no reference for \eqref{eq:fourier-error}
yet. Therefore, we present the proof for completeness. The proof for
error estimate
\eqref{eq:milstein-error} can be found in \cite{MR1335454,MR1214374}. The
presented error estimate \eqref{eq:wiktorsson-error} follows from the
original result in \cite{MR1843055} and is an improvement that we also prove
in the following. Finally, for the error estimate \eqref{eq:mronroe-error}
we refer to \cite{Mrongowius2021}.
\begin{proof} %
\begin{enumerate}[(i)]
    \item For the proof of \eqref{eq:fourier-error}, we first observe that
	$\Erw \big[ \big| \mathcal{I}_{(i,i)}(h)
	-\hat{\mathcal{I}}_{(i,i)}^{\FS,\pp}(h) \big|^2 \big] = 0$ for $i=1, \ldots, m$.
	Further, noting that all $a_k^i$ and $b_l^j$ for $k,l \in \mathbb{N}$ and
	$i,j=1,\ldots,m$ are i.i.d.\ Gaussian random variables with distribution
	given in \eqref{eq:fourier-coeffs-distribution} that are also independent
	from $W_h^q$ for $q=1, \ldots, m$, we calculate with
	\eqref{eq:expansion-levy-area} for $i \neq j$ that
	\begin{align}
		\Erw \Big[ \big| \mathcal{I}_{(i,j)}(h)
		-\hat{\mathcal{I}}_{(i,j)}^{\FS,\pp}(h) \big|^2\Big]
		&= \Erw \Big[ \Big| \pi \sum_{r=p+1}^{\infty} r
		\Big(a_r^i \Big( b_r^j -\frac{1}{\pi r} W_h^j \Big)
		-\Big( b_r^i -\frac{1}{\pi r} W_h^i \Big) a_r^j \Big) \Big|^2 \Big] \nonumber \\
		&= 2\pi^2 \sum_{r=p+1}^{\infty} r^2 \,
		\Erw \Big[ \big( a_r^{i} \big)^2 \Big] \,
		\Erw \Big[ \Big( b_r^j -\frac{1}{\pi r} W_h^j \Big)^2 \Big] \nonumber \\
		&= \frac{3h^2}{2\pi^2} \sum_{r=p+1}^{\infty} \frac{1}{r^2}
		\, . \label{Eq-max-L2-error-FS-exact}
	\end{align}
	Then, the error estimate \eqref{eq:fourier-error} follows with the estimate
	\begin{equation*} %
		\sum_{r=p+1}^{\infty} \frac{1}{r^2}
		\leq \int_{p}^{\infty} \frac{1}{r^2} \differential r
		= \frac{1}{p} \, .
	\end{equation*}
    \setcounter{enumi}{2}
    \item In order to prove \eqref{eq:wiktorsson-error}, we proceed analogously to \cite{Mrongowius2021}. Instead of \cite[Lemma 4.1]{MR1843055} we employ the variant \cite[Lemma 4.3]{Mrongowius2021} and note that one can prove the following stronger version of \cite[Theorem 4.2]{MR1843055} similar to \cite[Proposition 4.2]{Mrongowius2021}
    \begin{align*}
        &\sum_{q=1}^{M} \Erw\Big[\big\lvert (\Sigma^{\pp}-\Sigma^\infty)_{r,q} \big\rvert^2\,\Big|\, W_h\Big] \\
        &\quad= \frac{\sum_{k=p+1}^{\infty}1/k^4}{(2\sum_{k=p+1}^{\infty}1/k^2)^2}
        \cdot 2\bigg(m+m\frac{(W_h^i)^2}{h}+m\frac{(W_h^j)^2}{h}+2\frac{\lVert W_h\rVert^2}{h}\bigg)
    \end{align*}
    for all $ 1\leq r\leq M $ where $ 1\leq j<i\leq m $ with $ (j-1)(m-\frac{j}{2})+i-j=r $.
    Recognizing that $ \Erw\Big[ m+m\frac{(W_h^i)^2}{h}+m\frac{(W_h^j)^2}{h}+2\frac{\lVert W_h\rVert^2}{h} \Big] = 5m $ and using that $ \frac{\sum_{k=p+1}^{\infty}1/k^4}{(\sum_{k=p+1}^{\infty}1/k^2)^2} \leq \frac{1}{3p} $ (see also the proof of \cite[Theorem 4.2]{MR1843055}) completes the proof.
\end{enumerate}
\end{proof}
Here, it has to be pointed out that the Milstein algorithm is asymptotically
optimal in the class of algorithms that only make use of linear functionals of the
Wiener processes like, e.g., the Fourier algorithm, see \cite{MR2352954}.
Algorithms that belong to this class have also been considered in the
seminal paper by Clark and Cameron, see \cite{MR609181}, where as a consequence
it is shown that the same order of convergence as for the Milstein algorithm can be
obtained if the Euler--Maruyama scheme is applied for approximating
the iterated stochastic integrals.
On the other hand, the algorithms by Wiktorsson and by Mrongowius and
R\"o{\ss}ler do not belong to this class as they make also use of non-linear
information about the Wiener process. That is why these two algorithms allow
for a higher order of convergence w.r.t.\ the parameter $p$. While the
algorithm by Wiktorsson is based on the Fourier algorithm approach,
the improved algorithm by Mrongowius and R\"o{\ss}ler is build on the
asymptotically optimal approach of the Milstein algorithm. As a result of this,
the algorithm by Mrongowius and R\"o{\ss}ler allows for a smaller error
constant by a factor $1/{\sqrt{5}}$ compared to the Wiktorsson algorithm.

Next to the error estimates in $L^2(\Omega)$-norm presented in
Theorem~\ref{thm:approx-error-estimates}, it is worth mentioning that
there also exist error estimates in general $L^q(\Omega)$-norm with $q>2$
for the Milstein algorithm and for the Mrongowius--R\"o{\ss}ler algorithm, see
\cite[Prop. 4.6, Thm. 4.8]{Mrongowius2021}.

\begin{algorithm}[htb]
	\normalsize
	\caption{Calculation of the iterated integrals}
	\label{alg:iterated-integrals}
	\begin{algorithmic}[1]
		\Require{$ W_h $ \Comment{m-dimensional Wiener increment [vector of length m]}}
		\Require{$ h $ \Comment{step size of the increment [scalar]}}
		\Require{$ \Se $ \Comment{truncation parameter for tail sum [natural number]}}
		\Statex

		\Function{iterated\textunderscore integrals}{$ W_h, h, \Se $}
		\Comment{based on the error estimates}
		\State $ W_{\mathrm{std}}\gets \frac{1}{\sqrt{h}}W_h $ \Comment{standardisation}
		\State $ \LevyArea\gets $\Call{levy\textunderscore area}{$ W_{\mathrm{std}},\Se $}
		 \Comment{call algorithm for L\'{e}vy area approximation}
		\State $ \IterInt_{\mathrm{std}}\gets \frac{1}{2}W_{\mathrm{std}}W_{\mathrm{std}}^\top
		- \frac{1}{2}I_m + \LevyArea $
		\State $ \IterInt\gets h\cdot\IterInt_{\mathrm{std}} $
		\State \Return $ \IterInt $ \Comment{return matrix with iterated stoch.\ integrals}
		\EndFunction
	\end{algorithmic}
\end{algorithm}
For the simulation of the iterated stochastic integrals
$\hat{\mathcal{I}}^{\Alg,\pp}$ by making use of \eqref{eq:def-hat-I-Alg-p}
and by applying one of the presented algorithms ($\Alg$)
in Sections~\ref{subsec:alg-fourier-series}--\ref{subsec:alg-mronroe}
for the approximate simulation of the L\'{e}vy areas we
refer to the pseudocode of Algorithm~\ref{alg:iterated-integrals}.
That is, given a realization of an increment of the $m$-dimensional Wiener process
$W_h$ w.r.t.\ step size $h>0$ and a value of the truncation parameter $p \in \mathbb{N}$
for the tail sum of the L\'{e}vy area approximation as input parameters,
Algorithm~\ref{alg:iterated-integrals} calculates and returns the
matrix $\hat{\mathcal{I}}^{\Alg,\pp}(h)$ based on the algorithm
$\Call{levy\textunderscore area}$ for the L\'{e}vy area approximation
that has to be replaced by the choice of any of the algorithms described
in Sections~\ref{subsec:alg-fourier-series}--\ref{subsec:alg-mronroe}.

If iterated stochastic integrals are needed for numerical approximations
of, e.g., solutions of SDEs in the root mean square sense, then one has to
simulate these iterated stochastic integrals with sufficient accuracy in
order to preserve the overall convergence rate of the approximation scheme.
For SDEs, let $\gamma >0$ denote the order of convergence in the $L^2$-norm of
the numerical scheme under consideration, e.g., $\gamma=1$
for the well known Milstein scheme \cite{MR1214374,MR1335454} or a
corresponding stochastic Runge--Kutta scheme \cite{MR2669396}.
Then, the required order of convergence for the approximation
of some random variable like, e.g., an iterated stochastic integral,
that occurs in a given numerical scheme for SDEs
such that its order $\gamma$ is preserved has been established in
\cite[Lem.~6.2]{MR1335454} and \cite[Cor.~10.6.5]{MR1214374}.
Since it holds $\Erw \big[ \mathcal{I}_{(i,j)}(h)
- \hat{\mathcal{I}}^{\Alg,\pp}_{(i,j)}(h) \big] = 0$ for
$ 1\leq i,j\leq m $ for all algorithms under consideration in
Section~\ref{sec:algorithms}, the iterated
stochastic integral approximations
need to fulfil
\begin{equation}\label{eq:integralerrorbound}
	\big\| \mathcal{I}_{(i,j)}(h)
	-\hat{\mathcal{I}}^{\Alg,\pp}_{(i,j)}(h) \big\|_{L^2(\Omega)}
	\leq c_1 \, h^{\gamma+\frac{1}{2}}
\end{equation}
for all $ 1\leq i,j\leq m $, all sufficiently small $h>0$ and
some constant $c_1 > 0$ independent of $h$.
This condition can equivalently be expressed using the \maxL-norm as
\begin{equation}\label{eq:integralerrorbound-matrix}
    \big\| \mathcal{I}(h) - \hat{\mathcal{I}}^{\Alg,\pp}(h)
    \big\|_\maxL \leq c_1 \, h^{\gamma+\frac{1}{2}} \,.
\end{equation}
Together with Theorem~\ref{thm:approx-error-estimates} this
condition specifies how to choose the truncation parameter $\Se$.

For the Fourier algorithm (Alg.~\ref{alg:fourier}) or the Milstein
algorithm (Alg.~\ref{alg:milstein}) we have to choose $\Se$
such that
\begin{equation} \label{eq:fourier-cutoff}
	\Se \geq \frac{c_2}{c_1^2} \cdot h^{1-2\gamma}
	\in \mathcal{O} \big( h^{1-2\gamma} \big)
\end{equation}
with $c_2=\frac{3}{2\pi^2}$ for the Fourier algorithm and
$c_2=\frac{1}{2\pi^2}$ for the Milstein algorithm.

On the other hand, for the Wiktorsson algorithm (Alg.~\ref{alg:wiktorsson})
or the Mrongowius--R\"o{\ss}ler algorithm (Alg.~\ref{alg:mronroe})
we can choose $p$ such that
\begin{equation} \label{eq:wik-cutoff}
	\Se \geq \frac{\sqrt{c_3}}{c_1} \cdot h^{\frac{1}{2}-\gamma}
	\in \mathcal{O} \big( h^{\frac{1}{2}-\gamma} \big)
\end{equation}
where $c_3 = \frac{5m}{12 \pi^2}$ for the Wiktorsson algorithm and
$c_3 = \frac{m}{12 \pi^2}$ for the Mrongowius--R\"o{\ss}ler algorithm.

As a result of this, we have to choose $\Se \in \mathcal{O} \big( h^{-1} \big)$
if Algorithm~\ref{alg:fourier} or Algorithm~\ref{alg:milstein} is
applied and $\Se \in \mathcal{O} \big( h^{-\frac{1}{2}} \big)$ if
Algorithm~\ref{alg:wiktorsson} or Algorithm~\ref{alg:mronroe} is
applied in the Milstein scheme or any other numerical scheme for SDEs
which has an order of convergence
$\gamma=1$ in the $L^2$-norm and thus also strong order of convergence
$\gamma=1$. Hence, the algorithm by Wiktorsson as well as the algorithm
by Mrongowius and R\"o{\ss}ler typically need a significantly lower value of the truncation
parameter $\Se$ compared to the Fourier algorithm and the Milstein algorithm.
Additionally, for the algorithm by Mrongowius and R\"o{\ss}ler the parameter
$p$ can be chosen to be by a factor $1/\sqrt{5}$ smaller than for the algorithm
by Wiktorsson.
\section{Comparison of the algorithms and their performance}
\label{sec:theo-comp}
In this section we will look at the costs of each algorithm to achieve a
given precision. In particular we will see how to choose the optimal parameters for
each algorithm in a given setting. After that we show the results
of a simulation study to confirm the theoretical order of convergence for
each algorithm.
\subsection{Comparison of computational cost}
\label{subsec:costs}
We will measure the computational cost in terms of the number of
random numbers drawn from a standard Gaussian distribution that have to be
generated for the simulation of all twice iterated stochastic
integrals corresponding to one given
increment of an underlying $m$-dimensional Wiener process.
For a fixed truncation parameter $p$ these can be easily counted in the
derivation or in the pseudocode of the algorithms presented above.
E.g., Algorithm~\ref{alg:fourier} needs $2pm$ random numbers
corresponding to the Fourier coefficients $\alpha_r^i$ and $\beta_r^i$
for $1 \leq r \leq p$ and $1 \leq i \leq m$.
For each algorithm these costs are summarised in
Table~\ref{tab:random-numbers}.
\begin{table}%
	\renewcommand{\arraystretch}{1.3}
	\centering
	\begin{tabular}{@{}ll@{}}
		\toprule
		Algorithm & \# $\Normal{0}{1}$ random numbers \\
		\midrule
		Fourier & $ 2\Se m $ \\
		Milstein & $ 2\Se m+m $ \\
		Wiktorsson & $ 2\Se m+\frac{m^2-m}{2} $ \\
		Mrongowius--Rößler \quad & $ 2\Se m+\frac{m^2-m}{2}+m $ \\
		\bottomrule
	\end{tabular}
	\caption{Number of random numbers that need to be generated
		in terms of the dimension~$m$ of the Wiener process and the
		truncation parameter~$ \Se $.}
	\label{tab:random-numbers}
\end{table}

In practice however, a fixed precision $\bar{\varepsilon}$ is usually
given with respect to some norm for the error
and we are interested in the minimal cost to simulate the iterated stochastic
integrals to this precision.
For this we need to choose the truncation parameter
$p = p(m,h,\bar{\varepsilon})$ that may depend
on the dimension~$m$, the step size~$h$ and the desired precision~$\bar{\varepsilon}$,
as discussed in Section~\ref{sec:error-estimates}, and insert this
expression for $p$ into the cost in Table~\ref{tab:random-numbers}. Again,
considering for example the Fourier algorithm (Alg.~\ref{alg:fourier}),
we calculate from \eqref{eq:fourier-error} that the truncation parameter
has to be chosen as $p \geq \frac{3h^2}{2\pi^2 \bar{\varepsilon}^2}$
if we measure the error in the \maxL-norm. This allows us to calculate
the actual cost to be $c(m,h,\bar{\varepsilon}) = 2pm
= \frac{3h^2m}{\pi^2 \bar{\varepsilon}^2}$. If instead we want to bound
the error in the \LF-norm, we have to
choose $p \geq \frac{3h^2(m^2-m)}{2\pi^2 \bar{\varepsilon}^2}$
with the cost of $c(m,h, \bar{\varepsilon})
= \frac{3h^2(m^3-m^2)}{\pi^2 \bar{\varepsilon}^2}$.

On the other hand, sometimes we are willing to invest a specific
amount of computation time or cost and are interested in the
minimal error we can achieve without exceeding this given cost
budget. Rearranging the expression from above for the cost, we find
that given some fixed cost budget $\bar{c}>0$ we can get an upper bound on
the \maxL-error for the Fourier algorithm
of $\varepsilon(m,h,\bar{c})
= \frac{\sqrt{3} h \sqrt{m}}{\pi \sqrt{\bar{c}}}$.
In Table~\ref{tab:theoreticalcomparison},
we list the lower bound for the truncation parameter $p(m,h, \bar{\varepsilon})$,
the resulting cost $c(m,h, \bar{\varepsilon})$ and the minimal
achievable error $\varepsilon(m,h,\bar{c})$
for a prescribed error $\bar{\varepsilon}$ or a
given cost budget $\bar{c}$ for all four algorithms and
two different error criteria introduced in
Section~\ref{subsec:relation-matrix-norms}.
\begin{table}%
	\renewcommand{\arraystretch}{1.7}
	\centering
	\begin{tabular}{@{}llll@{}}
		\toprule
		Algorithm & Objective & \multicolumn{2}{c}{Applied error criterion} \\
		\cmidrule{3-4}
		&& $\|\cdot\|_\maxL$ & $\|\cdot\|_\LF$ \\
		\midrule
		Fourier &&& \\
		& cut-off \ $ p(m,h, \bar{\varepsilon}) $ & $ \frac{3h^2}{2\pi^2 \bar{\varepsilon}^2} $ & $ \frac{3h^2(m^2-m)}{2\pi^2 \bar{\varepsilon}^2} $ \\
		& cost \ $ c(m,h, \bar{\varepsilon}) $ & $ \frac{3h^2m}{\pi^2 \bar{\varepsilon}^2} $ & $ \frac{3h^2(m^3-m^2)}{\pi^2 \bar{\varepsilon}^2} $ \\
		& error \ $ \varepsilon(m,h,\bar{c}) $ & $ \frac{\sqrt{3}}{\pi}\cdot\frac{h\cdot\sqrt{m}}{\sqrt{\bar{c}}} $ & $ \frac{\sqrt{3}}{\pi}\cdot\frac{h\cdot m\sqrt{m-1}}{\sqrt{\bar{c}}} $ \\
		Milstein &&& \\
		& cut-off \ $ p(m,h, \bar{\varepsilon}) $ & $ \frac{h^2}{2\pi^2 \bar{\varepsilon}^2} $ & $ \frac{h^2(m^2-m)}{2\pi^2 \bar{\varepsilon}^2} $ \\
		& cost \ $ c(m,h, \bar{\varepsilon}) $ & $ \frac{h^2m}{\pi^2 \bar{\varepsilon}^2}+m $ & $ \frac{h^2(m^3-m^2)}{\pi^2 \bar{\varepsilon}^2}+m $ \\
		& error \ $ \varepsilon(m,h,\bar{c}) $ & $ \frac{1}{\pi}\cdot\frac{h\cdot\sqrt{m}}{\sqrt{\bar{c}-m}} $ & $ \frac{1}{\pi}\cdot\frac{h\cdot m\sqrt{m-1}}{\sqrt{\bar{c}-m}} $ \\
		Wiktorsson &&& \\
		& cut-off \ $ p(m,h, \bar{\varepsilon}) $ & $ \frac{\sqrt{5}h\sqrt{m}}{\sqrt{12}\pi \bar{\varepsilon}} $ & $ \frac{\sqrt{5}hm\sqrt{m-1}}{\sqrt{12}\pi \bar{\varepsilon}} $ \\
		& cost \ $ c(m,h, \bar{\varepsilon}) $ & $ \frac{\sqrt{5}hm^{3/2}}{\sqrt{3}\pi \bar{\varepsilon}}+\frac{m^2-m}{2} $ & $ \frac{\sqrt{5}hm^2\sqrt{m-1}}{\sqrt{3}\pi \bar{\varepsilon}}+\frac{m^2-m}{2} $ \\
		& error \ $ \varepsilon(m,h,\bar{c}) $ & $ \frac{2\sqrt{5}}{\sqrt{3}\pi}\cdot\frac{h\cdot m^{3/2}}{2 \bar{c} - m^2 + m} $ & $ \frac{2\sqrt{5}}{\sqrt{3}\pi}\cdot\frac{h\cdot m^2\sqrt{m-1}}{2 \bar{c} - m^2 + m} $ \\
		Mrongowius--R\"o{\ss}ler &&& \\
		& cut-off \ $ p(m,h, \bar{\varepsilon}) $ & $ \frac{h\sqrt{m}}{\sqrt{12}\pi \bar{\varepsilon}} $ & $ \frac{hm\sqrt{m-1}}{\sqrt{12}\pi \bar{\varepsilon}} $ \\
		& cost \ $ c(m,h, \bar{\varepsilon}) $ & $ \frac{hm^{3/2}}{\sqrt{3}\pi \bar{\varepsilon}}+\frac{m^2+m}{2} $ & $ \frac{hm^2\sqrt{m-1}}{\sqrt{3}\pi \bar{\varepsilon}}+\frac{m^2+m}{2} $ \\
		& error \ $ \varepsilon(m,h,\bar{c}) $ & $ \frac{2}{\sqrt{3}\pi}\cdot\frac{h\cdot m^{3/2}}{2 \bar{c} - m^2 - m} $ & $ \frac{2}{\sqrt{3}\pi}\cdot\frac{h\cdot m^2\sqrt{m-1}}{2 \bar{c} - m^2 - m} $ \\
		\bottomrule
	\end{tabular}
	\caption{Comparison of theoretical properties of
		the algorithms~\ref{alg:fourier}-\ref{alg:mronroe} in
		Section~\ref{sec:algorithms}. The cut-off $p(m,h,\bar{\varepsilon})$
		describes how to choose the truncation parameter $p$ in order to
		guarantee that the error in the corresponding norm is
		less than $\bar{\varepsilon}$. The cost $c(m,h, \bar{\varepsilon})$ is the number of
		random numbers that need to be generated to achieve
		this error. Finally, the error $\varepsilon(m,h,\bar{c})$
		is the minimal achievable error given a fixed cost budget
		$\bar{c}$. Note that $m$ and $h$ are the dimension of
		the Wiener process and the step size, respectively.}
	\label{tab:theoreticalcomparison}
\end{table}

Now we can calculate for concrete situations the cost
of each algorithm and determine the cheapest one.
Consider again the example of a numerical scheme for
SDEs with $L^2$-convergence of order $\gamma=1$.
In this situation we have seen before that we have to
couple the error and the step size as
$\bar{\varepsilon} = \mathcal{O}(h^{3/2})$
(cf. \eqref{eq:integralerrorbound-matrix}).
Using this coupling the cost only depends on the dimension,
the step size and the chosen error norm.
In Figure~\ref{fig:optimal-algorithm} we choose
$\bar{\varepsilon} = h^{3/2}$ and we visualize
for two of the random matrix norms introduced in
Section~\ref{subsec:relation-matrix-norms} the
algorithm with minimal cost subject to dimension
and step size.

\begin{figure}%
	\centering
	\includegraphics{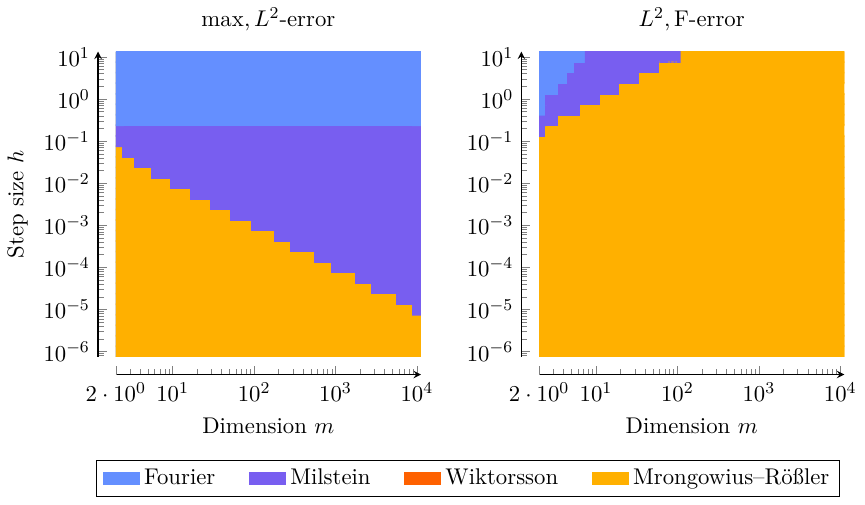}
	\caption{Optimal algorithm given the dimension $ m $
		and the step size $ h $ in the respective norm for
		$\bar{\varepsilon}=h^{3/2}$.}
	\label{fig:optimal-algorithm}
\end{figure}

The first thing we see is that the Wiktorsson algorithm
is never the best choice in these scenarios. This is because
for the Mrongowius--R\"o{\ss}ler algorithm the required value of the truncation
parameter can be reduced by a factor $ \sqrt{5} $ while only
requiring $ m $ additional random numbers.
Figure~\ref{fig:optimal-algorithm}
confirms that this trade-off is always worth it, i.e.,
the Mrongowius--R\"o{\ss}ler algorithm always
outperforms the Wiktorsson algorithm.

Moreover, for the \maxL-error we see in the left plot of
Figure~\ref{fig:optimal-algorithm} that if we fix some step size
then there always exists some sufficiently high dimension
such that it will be more efficient to apply the
Milstein algorithm.
This is a consequence of the higher order dependency on the
dimension of the Mrongowius--R\"o{\ss}ler algorithm (and also for the
Wiktorsson algorithm) compared to the Milstein algorithm which
dominates the error for large enough dimension if the step size
is fixed.
This means the Milstein algorithm is the optimal choice in this setting.

However, for any fixed dimension of the Wiener process
(as it is the case for, e.g., SDEs) there
always exists some threshold step size such that for all step sizes
smaller than this threshold
the Mrongowius--R\"o{\ss}ler algorithm
is more efficient than the other algorithms. This is due to the
higher order of convergence in time of the
Mrongowius--R\"o{\ss}ler algorithm. Thus, as the
step size tends to zero or if the step size is sufficiently small,
the Mrongowius--R\"o{\ss}ler algorithm is the optimal choice.

The situation in the right plot of Figure~\ref{fig:optimal-algorithm}
for the \LF-error is
different, because the simple scaling factor due to
Lemma~\ref{lem:norm-aequivalenz} does not translate into a
simple scaling of the associated cost. Indeed it is no simple
scaling at all if the cost contains additive terms independent
of the error.
Instead, only the error-dependent part of the
cost is scaled by $ m^2-m $ for the Fourier and Milstein
algorithms and by $ \sqrt{m^2-m} $ for the Wiktorsson and
Mrongowius--R\"o{\ss}ler algorithms (cf.~Table~\ref{tab:theoreticalcomparison}).
This is due to the different exponents of the truncation
parameter in the corresponding error estimates.
As a result of this, for the \LF-error the
Mrongowius--R\"o{\ss}ler algorithm is the optimal choice
except for the particular case of a rather big step size
and a low dimension of the Wiener process.
\subsection{A study on the order of convergence}
\label{subsec:convergence}
Convergence plots are an appropriate tool to complement theoretical
convergence results with numerical simulations. Unfortunately,
studying the convergence of random variables can be
difficult, especially if the target (limit) random variable that has
to be approximated can not be simulated exactly because its distribution is
unknown.
One way out of this problem is to substitute the target random
variable by a highly accurate approximate random variable that
allows to draw realizations from its known distribution.

In the present simulation study, we substitute the target random variable
$\mathcal{I}(h)$ by the approximate random variable
$\mathcal{I}^{\mref,( p_\mref )}(h)
= \hat{\mathcal{I}}^{\FS,( p_{\mref} )}(h)$
based on the fundamental and accepted Fourier algorithm with
high accuracy due to some large value for the parameter~$p_{\mref}$.
Therefore, we refer to $\mathcal{I}^{\mref,( p_\mref )}(h)$ as the
reference random variable and let $\hat{\mathcal{I}}^{\mathrm{Alg},\pp}(h)$
denote the approximation based on one of the algorithms described in
Section~\ref{sec:algorithms}.
We focus on the \maxL-error and the
\LF-error of the approximation algorithms that we want to compare.
We proceed as follows: 
Let $p_\mref \in \mathbb{N}$ be sufficiently large.
First, some realizations of $W_h$ and the reference random
variable $\mathcal{I}^{\mref,( p_\mref )}(h)$
are simulated.
Then, the corresponding realizations of the Fourier coefficients
$\alpha_r$ and $\beta_r$ for $1 \leq r \leq p_\mref$ are stored
together with $W_h$.
Next, the approximations
$\hat{\mathcal{I}}^{\mathrm{Alg},\pp}(h)$ are calculated
based on the same realization, which has to be done carefully.
The stored Fourier coefficients and the stored increment
of the Wiener process are used to compute
approximate solutions
using each of the algorithms with a truncation parameter $ p \ll p_\mref $.
Especially, we are able to exactly extract
from the stored Fourier coefficients and the stored increment of
the Wiener process
the corresponding realizations of the underlying standard Gaussian
random variables that are used in Wiktorsson's algorithm as well as
in the Mrongowius--R\"o{\ss}ler algorithm to approximate the tail.

To be precise, set $ h=1 $ and choose a large value $p_\mref$
for the reference random variable, for example $p_\mref=10^6$.
Now, simulate and store the Wiener increment $W_h$ as well
as the Fourier coefficients $\alpha_r^i$ and $\beta_r^i$
for $i=1,\ldots,m$ and $r=1,\ldots,p_{\mref}$. Then,
calculate the reference random variable
$\mathcal{I}^{\mathrm{ref},( p_\mref )}(h)$
using the Fourier algorithm.
Since every realization of the reference random variable
$\mathcal{I}^{\mathrm{ref}}(h)$ is also an approximation to
some realization of the random variable $\mathcal{I}(h)$,
we can control the precision in \maxL-norm of the reference
random variable exactly due to \eqref{Eq-max-L2-error-FS-exact}, i.e.,
it holds
\begin{align*}
	\big\| \mathcal{I}(h) - \mathcal{I}^{\mathrm{ref},( p_\mref )}(h)
	\big\|_{\maxL}  = \bigg( \frac{3 h^2}{2 \pi^2}
	\bigg( \frac{\pi^2}{6} - \sum_{r=1}^{p_\mref}
	\frac{1}{r^2} \bigg) \bigg)^{\frac{1}{2}} \, .
\end{align*}
For a large enough value of $p_\mref$ this error will be
sufficiently small such that it can be neglected in our
considerations. This justifies to work with the reference
random variable.

Next, choose a truncation value $p \ll p_\mref$ for the
approximation under consideration.
First, calculate $\hat{\mathcal{I}}^{\FS,(p)}(h)$
by the Fourier algorithm using only
the first $p$ stored Fourier coefficients. Then, extract
the necessary standard Gaussian vectors from the remaining
stored Fourier coefficients
using \eqref{eq:milstein-mu}, \eqref{eq:vec-R-p},
\eqref{eq:gamma-p-R-p} and \eqref{eq:RV-Psi-2-n-Exact}, i.e.,
calculate
\begin{align*}
    \gamma_1 &= \frac{1}{\sqrt{\psi_1(p+1)}}
    \sum_{r=p+1}^{p_\mref} \frac{1}{r} \,
    \alpha_r \,, \\
    \gamma &= \frac{1}{\sqrt{2\psi_1(p+1)}}
    \Big( \Sigma^\pp \Big)^{-1/2} K_m(P_m-I_{m^2})
    \sum_{r=p+1}^{p_\mref} \frac{1}{r} \Big( \alpha_r
    \otimes \big( \beta_r -\sqrt{2} W_h \big) \Big) \, , \\
    \gamma_2 &= \frac{1}{\sqrt{2\psi_1(p+1)}} \Big(
    \trunc{\Sigma_2} \Big)^{-1/2} K_m(P_m-I_{m^2})
    \sum_{r=p+1}^{p_\mref} \frac{1}{r}
    (\alpha_r \otimes \beta_r) \,.
\end{align*}
Calculating the vectors $\gamma_1$, $\gamma$ and $\gamma_2$
based on the already simulated and stored Fourier coefficients
guarantees that they fit correctly together with the realization
of the reference random variable. Note that for $r > p_\mref$
the Fourier coefficients $\alpha_r^i$ and $\beta_r^i$ for
$i=1, \ldots, m$ of the reference random variable are all zero.
Now, we calculate $\hat{\mathcal{I}}^{\Mil,(p)}(h)$
by the Milstein, $\hat{\mathcal{I}}^{\Wik,(p)}(h)$ by the Wiktorsson
and $\hat{\mathcal{I}}^{\MR,(p)}(h)$ by the Mrongowius--Rößler
algorithm using only the first $ p $ Fourier coefficients as well
as $ \gamma_1 $, $ \gamma $ and $ \gamma_2 $, respectively.

In Figure~\ref{fig:error-vs-effort} the computed
\maxL-errors are plotted versus the costs. Here, the
cost is determined as the number of random numbers necessary
to calculate the corresponding approximation for the
iterated stochastic integral for each algorithm. 
To generate the random numbers we use the Mersenne Twister generator
from Julia version 1.6.
The
steeper slope of the Wiktorsson and Mrongowius--R\"o{\ss}ler
algorithms can be clearly seen. This confirms the higher
order of convergence of these algorithms in contrast to
the Fourier and Milstein algorithms.
Furthermore, it is interesting to observe that the
Wiktorsson as well as the Mrongowius--R\"o{\ss}ler algorithm
perform better than their theoretical upper error bound given in
Theorem~\ref{thm:approx-error-estimates}.
This gives reason to suspect that the error estimates for both
algorithms given in
Theorem~\ref{thm:approx-error-estimates} are not sharp.

\begin{figure}[htb]
	\centering
	\includegraphics{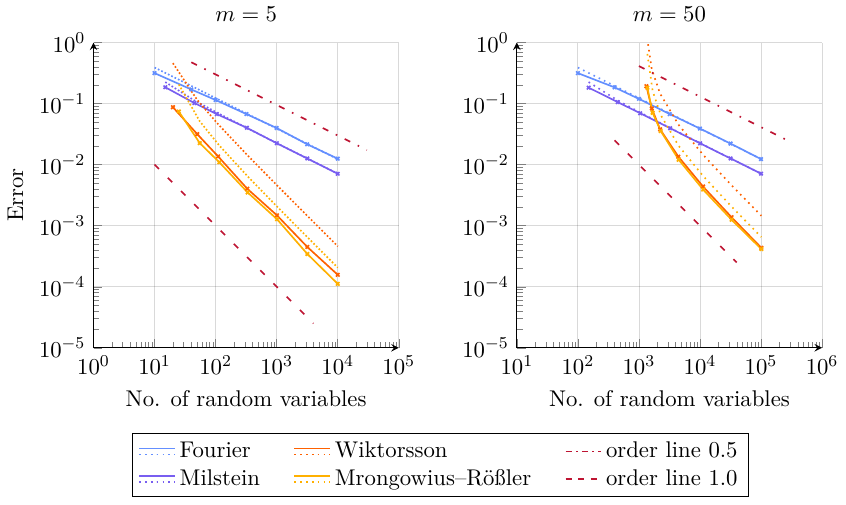}
	\caption{\maxL-error versus the number of random variables (cost) used for
		different dimensions of the Wiener process based on 100 realizations. The
		error is computed w.r.t.\ the reference random variable with
		$p_\mref = 10^6$ and for $ h=1 $. The dotted lines show the corresponding
        theoretical error bounds according to Theorem~\ref{thm:approx-error-estimates}.%
	}
	\label{fig:error-vs-effort}
\end{figure}
\section{A simulation toolbox for Julia and MATLAB}
\label{sec:software}
There exist several software packages to simulate SDEs
in various programming languages. A short search
brought up the following 20 toolboxes:
for the C++ programming language \cite{Avramidis2020,Janicki2003},
for the Julia programming language
\cite{Shardlow2016,Schauer2021,Akerlindh2019,StochasticDiffEq2021},
for Mathematica \cite{WolframItoProcess},
for \MATLAB{}
\cite{Picchini2017,Horchler2013,Gilsing2007,MatlabFinancial},
for the Python programming language
\cite{Ansmann2018,Aburn2021,Gevorkyan2016,Gevorkyan2018,torchsde2021}
and for the R programming language
\cite{Iacus2008,Brouste2014,Guidoum2020,Iacus2018}.
However, only four of these toolboxes seem to contain an implementation
of an approximation for the iterated stochastic integrals using
Wiktorsson's algorithm and none of them provide an implementation of
the recently proposed Mrongowius--R\"o{\ss}ler algorithm.
To be precise, Wiktorsson's algorithm is contained in the software
packages \package{SDELab}~\cite{Gilsing2007} in Matlab as well as in
its Julia continuation \package{SDELab2.jl}~\cite{Shardlow2016}, in
the Python package \package{sdeint}~\cite{Aburn2021} and more recently
also in the Julia package \package{StochasticDiffEq.jl}~\cite{StochasticDiffEq2021}.

Note that among these four software packages only three have the
source code readily available on GitHub:
\begin{itemize}
    \item The software package \package{SDELab2.jl} does not seem to be
    maintained any more. It is worth mentioning that its implementation
    avoids the explicit use of Kronecker products. However, it is not as
    efficient as the implementations discussed in
    Section~\ref{subsubsec:wiktorsson-implementation}. Moreover, some simple
    test revealed that the implementation seems to suffer from some typos
    that may cause inexact results. Additionally, the current version of this
    package does not run with the current Julia version~1.6.
    \item The software package \package{sdeint} is purely written in Python
    and as such it is inherently slower than comparable implementations in
    a compiled language. Furthermore, it is based on a rather na\"ive
    implementation, e.g., explicitly constructing the Kronecker products
    and the permutation and selection matrices as described in
    Section~\ref{subsubsec:wiktorsson-derivation}. However, this is
    computational costly and thus rather inefficient, see also the discussion
    in Section~\ref{subsubsec:wiktorsson-implementation}.
    \item The software package \package{StochasticDiffEq.jl} offers the
    latest implementation of Wiktorsson's approximation in the high-performance
    language Julia. Up to now, it is actively maintained and has a strong
    focus on good performance. It is basically a slightly optimized
    version of the \package{SDELab2.jl} implementation and thus it also
    does not achieve the maximal possible efficiency, especially compared
    to the discussion in Section~\ref{subsubsec:wiktorsson-implementation}.
\end{itemize}

In summary, there is currently no package providing an efficient algorithm
for the simulation of iterated integrals. That is the main reason we
provide a new simulation toolbox for Julia and \MATLAB{} that, among others,
features Wiktorsson's algorithm as well as the Mrongowius--R\"o{\ss}ler algorithm.
In Section~\ref{subsec:features} we present some more features of the toolbox.
After giving usage examples in Section~\ref{subsec:examples}, we analyse
the performance of our toolbox as compared to some existing implementations
in Section~\ref{subsec:benchmark}.
\subsection{Features of the new Julia and MATLAB simulation toolbox}
\label{subsec:features}
We introduce a new simulation toolbox for the simulation of twofold iterated
stochastic integrals and the corresponding L\'{e}vy areas for the Julia, see also \cite{MR3605826}, and
\MATLAB{} programming languages.
The aim of the toolbox is to provide both high performance and ease of use.
On the one hand, experts can directly control each part of the algorithms.
On the other hand, our software can automatically choose omitted parameters.
Thus, a non-expert user only has to provide the Wiener increment and
the associated step size and all other parameters will be chosen in an
optimal way.

The toolbox is available
as the software packages \package{LevyArea.jl} and \package{LevyArea.m} for Julia
and \MATLAB, respectively. Both software packages are freely available 
from Netlib\footnote{\url{http://www.netlib.org/numeralgo/}} as the \texttt{na57} package
and from
GitHub \cite{levyarea-jl-zenodo,levyarea-m-zenodo}. Additionally, the Julia package is registered in the "General"
registry of Julia and can be installed using the built-in package manager,
while the \MATLAB\ package is listed on the "\MATLAB\ Central - File Exchange"
and can be installed using the Add-On Explorer.
This allows for an easy integration of these packages into other
software projects.

Both packages provide the Fourier algorithm, the Milstein algorithm, the Wiktorsson
algorithm and, to the best of our knowledge, for the first time the
Mrongowius--R\"o{\ss}ler algorithm.
As the most important feature, all four algorithms are implemented following
the insights from Section~\ref{sec:algorithms} to provide fast and highly
efficient implementations.
In Section~\ref{subsec:benchmark} the performance of the Julia
package is analysed in comparison to existing software.

Additional features include the ability to automatically choose the
optimal algorithm based on the costs listed in
Table~\ref{tab:theoreticalcomparison}.
That is, given the increment of the driving Wiener process,
the step size and an error bound in some norm, both software packages
will automatically determine the optimal algorithm and the associated
optimal value for the truncation parameter.
Recall that optimal always means in terms of the number of
random numbers that have to be generated in order to obtain minimal
computing time.
Thus, using the option for an optimal choice of the algorithm
results in the best possible performance for each setting.
Especially, for the simulation of SDE solutions with some strong
order $ 1 $ numerical scheme, both software packages can determine
all necessary parameters by passing only the Wiener increment and
the step size to the toolbox.
Everything else is determined automatically such that the global
order of convergence of the numerical integrator is preserved.

Furthermore, it is possible to directly simulate iterated stochastic integrals
based on $Q$-Wiener processes on finite-dimensional spaces as they typically
appear for the approximate simulation of solutions to SPDEs. In that case, the
eigenvalues of the covariance operator $Q$ need to be passed
to the software toolbox in order to compute the correctly scaled iterated
stochastic integrals. This scaling is briefly described at the
end of Section~\ref{subsec:iterint-levyarea}, see also
\cite{MR3949104} for a detailed discussion.
\subsection{Usage of the software package}
\label{subsec:examples}
Next, we give some basic examples how to make use of the two software packages
\package{LevyArea.jl} for Julia and \package{LevyArea.m} for \MATLAB.
The aim is not to give a full documentation but rather to show some
example invocations of the main functions.
For more information we refer to the documentation that comes with
the software and can be easily accessed in Julia and \MATLAB.
\subsubsection{The Julia package \package{LevyArea.jl}}
\label{subsubsec:examples-julia}
The following code works with Julia version~1.6.
First of all, the software package \package{LevyArea.jl} \cite{levyarea-jl-zenodo} needs to
be installed to make it available. Therefore, start Julia and
enter the package manager by typing \jlinl{]} (a closing square bracket).
Then execute

{\Large\begin{jllisting}
pkg> add LevyArea
\end{jllisting}}
\noindent
which downloads the package and adds it to the current project.
After the installation of the package, one can load the package
and initialize some variables that we use in the following:

{\Large\begin{jllisting}
julia> using LevyArea
julia> m = 5; # dimension of Wiener process
julia> h = 0.01; # step size or length of time interval
julia> err = 0.05; # error bound
julia> W = sqrt(h) * randn(m); # increment of Wiener process
\end{jllisting}}

\noindent
Here, $W$ is the $m$-dimensional vector of increments of the driving
Wiener process on some time interval of length $h$.

For the simulation of the corresponding twofold iterated stochastic
integrals, one can use the following default call of the function
\jlinl{iterated_integrals} where only the increment and the
step size are mandatory:

{\Large\begin{jllisting}
julia> II = iterated_integrals(W,h)
\end{jllisting}}

\noindent
In this example, the error $\bar{\varepsilon}$ is not explicitly specified.
Therefore, the
function assumes the desired precision to be $\bar{\varepsilon}=h^{3/2}$
as it has to be chosen for the numerical solution of SDEs,
see end of Section~\ref{sec:error-estimates},
and automatically chooses the optimal algorithm according to the logic in
Section~\ref{subsec:costs}, see also Figure~\ref{fig:optimal-algorithm}.
If not stated otherwise, the default error criterion is the \maxL-error
and the function returns the $m \times m$ matrix \jlinl{II} containing a realization
of the approximate iterated stochastic integrals that correspond
to the given increment $W$.

The desired precision $\bar{\varepsilon}$ can be optionally provided
using a third positional argument:

{\Large\begin{jllisting}
julia> II = iterated_integrals(W,h,err)
\end{jllisting}}

\noindent
Again, the software package automatically chooses the optimal
algorithm as analysed in Section~\ref{subsec:costs}.

To determine which algorithm is chosen by the package without simulating any iterated
stochastic integrals yet, the function \jlinl{optimal_algorithm} can
be used. The arguments to this function are the dimension of the Wiener
process, the step size and the desired precision:

{\Large\begin{jllisting}
julia> alg = optimal_algorithm(m,h,err); # output: Fourier()
\end{jllisting}}

It is also possible to choose the algorithm directly using the
keyword argument \jlinl{alg}. The value can be one of
\jlinl{Fourier()}, \jlinl{Milstein()}, \jlinl{Wiktorsson()} and \jlinl{MronRoe()}:

{\Large\begin{jllisting}
julia> II = iterated_integrals(W,h; alg=Milstein())
\end{jllisting}}

As the norm for the considered error, e.g., the \maxL- and \LF-norm
can be selected using a keyword argument. The accepted
values are \jlinl{MaxL2()} and \jlinl{FrobeniusL2()}:

{\Large\begin{jllisting}
julia> II = iterated_integrals(W,h,err; error_norm=FrobeniusL2())
\end{jllisting}}

If iterated stochastic integrals for some $Q$-Wiener process need to
be simulated, like for the numerical simulation of solutions to SPDEs,
then the increment of the $Q$-Wiener process together with the
square roots of the eigenvalues of the associated covariance
operator have to be provided, see Section~\ref{subsec:iterint-levyarea}:

{\Large\begin{jllisting}
julia> q = [1/k^2 for k=1:m]; # eigenvalues of cov. operator
julia> QW = sqrt(h) * sqrt.(q) .* randn(m); # Q-Wiener increment
julia> IIQ = iterated_integrals(QW,sqrt.(q),h,err)
\end{jllisting}}

\noindent
In this case, the function \jlinl{iterated_integrals} utilizes a
scaling of the iterated stochastic integrals as explained in
Section~\ref{subsec:iterint-levyarea} and also adjusts the error
estimates appropriately such that the error bound holds w.r.t.\ the
iterated stochastic integrals $\mathcal{I}^{Q}(h)$ based on the
$Q$-Wiener process. Here the error norm defaults to the \LF-error.

Note that all discussed keyword arguments are optional and can be
combined as favoured. Additional information can be found using the
Julia help mode:

{\Large\begin{jllisting}
julia> ?iterated_integrals
julia> ?optimal_algorithm
\end{jllisting}}

\subsubsection{The MATLAB package \package{LevyArea.m}}
\label{subsubsec:examples-matlab}
The following code works with \MATLAB\ version~2020a.
The installation of the software package \package{LevyArea.m} \cite{levyarea-m-zenodo} in
\MATLAB\ is done either by copying the package folder \texttt{+levyarea}
into the current working directory or by installing the \MATLAB\
toolbox file \texttt{LevyArea.mltbx}.
This can also be done through the Add-On Explorer.

The main function of the toolbox is the function
\jlinl{iterated_integrals}. It can be called by prepending the
package name \jlinl{levyarea.iterated_integrals}. However, since
this may be cumbersome one can import the used functions once by

{\Large\begin{jllisting}
>> import levyarea.iterated_integrals
>> import levyarea.optimal_algorithm
\end{jllisting}}

\noindent
and then one can omit the package name by simply calling
\jlinl{iterated_integrals} or \jlinl{optimal_algorithm},
respectively.
In the following, we assume that the two functions are imported, so that we can always call them directly without the package name.

For the examples considered next we initialize
some auxiliary variables:

{\Large\begin{jllisting}
>> m = 5;	
>> h = 0.01;	
>> err = 0.05;	
>> W = sqrt(h) * randn(m,1);	
\end{jllisting}}

\noindent
Note that $W$ denotes the $m$-dimensional vector of increments of the
Wiener process on some time interval of length $h$.

For directly simulating the twofold iterated stochastic integrals
given the increment of the Wiener process, one can call the main
function \jlinl{iterated_integrals} passing the increment and
the step size:

{\Large\begin{jllisting}
>> II = iterated_integrals(W,h)
\end{jllisting}}

\noindent
These two parameters are mandatory. In this case, the precision
is set to the default $\bar{\varepsilon}=h^{3/2}$ and the optimal
algorithm is applied automatically according to Section~\ref{subsec:costs} (see also Figure~\ref{fig:optimal-algorithm}). The default norm for the
error is set to be the \maxL-norm.

The desired precision $\bar{\varepsilon}$ can be optionally provided using a third positional argument:

{\Large\begin{jllisting}
>> II = iterated_integrals(W,h,err)
\end{jllisting}}

\noindent
Here, again the software package automatically chooses the
optimal algorithm as analysed in Section~\ref{subsec:costs}.

In order to determine which algorithm is optimal for some
given parameters without simulating the iterated stochastic integrals
yet, the function \jlinl{optimal_algorithm} can be used:

{\Large\begin{jllisting}
>> alg = optimal_algorithm(m,h,err); 
\end{jllisting}}

\noindent
The arguments to this function are the dimension of the Wiener
process, the step size and the desired precision.

On the other hand, it is also possible to choose the used
algorithm directly using a key-value pair. The value can be
one of \jlinl{'Fourier'}, \jlinl{'Milstein'}, \jlinl{'Wiktorsson'}
and \jlinl{'MronRoe'}. E.g., to use the Milstein algorithm call:

{\Large\begin{jllisting}
>> II = iterated_integrals(W,h,'Algorithm','Milstein')
\end{jllisting}}

The desired norm for the prescribed error bound can also be
selected using a key-value pair. The accepted values are
\jlinl{'MaxL2'} and \jlinl{'FrobeniusL2'} for the \maxL- and
\LF-norm, respectively. E.g., in order to use the \LF-norm call:

{\Large\begin{jllisting}
>> II = iterated_integrals(W,h,err,'ErrorNorm','FrobeniusL2')
\end{jllisting}}

The simulation of numerical solutions to SPDEs often requires
iterated stochastic integrals based on $Q$-Wiener processes.
In that case, the square roots of the eigenvalues of the associated
covariance operator need to be provided. Therefore, first define
all necessary variables and then call the function
\jlinl{iterated_integrals} using the key \jlinl{'QWiener'}
as follows:

{\Large\begin{jllisting}
>> q = 1./(1:m)'.^2; 
>> QW = sqrt(h) * sqrt(q) .* randn(m,1); 
>> IIQ = iterated_integrals(QW,h,err,'QWiener',sqrt(q))
\end{jllisting}}

\noindent
In this case, the function utilizes the scaling of the iterated
stochastic integrals as explained in Section~\ref{subsec:iterint-levyarea}
and it also adjusts the error estimates w.r.t.\ $\mathcal{I}^Q(h)$
appropriately.

Note that all discussed keyword arguments (key-value pairs) are
optional and can be combined as desired. Additional information
can be found using the \jlinl{help} function:

{\Large\begin{jllisting}
>> help iterated_integrals
>> help levyarea.optimal_algorithm
\end{jllisting}}

\subsection{Benchmark comparison}
\label{subsec:benchmark}
To assess the performance of the new software package, we
compare it with some existing implementations. Therefore,
we consider the software packages \package{SDELab2.jl} \cite{Shardlow2016},
\package{sdeint} \cite{Aburn2021} and \package{StochasticDiffEq.jl} \cite{StochasticDiffEq2021}.
First, the software package \package{sdeint} is completely
written in Python and a quick simulation example with $m=50$,
$h=0.01$, $\varepsilon=0.001$ and thus $p=15$ takes $9.9$ seconds
to generate the matrix of iterated integrals with Wiktorsson's
algorithm. For the same parameters, our implementation of
Wiktorsson's algorithm in the package \package{LevyArea.jl}
completed in only $ 5.5\cdot 10^{-6} $ seconds. This is a
speed-up by a factor $1.8 \times 10^6$ and therefore we exclude the
Python package \package{sdeint} from our further comparison.
Moreover, since the package \package{SDELab2.jl} seems to be unmaintained
and does not run unmodified on current Julia versions, it is excluded
from our benchmark as well. Thus, we compare the new package
\package{LevyArea.jl} with the implementation in the
package \package{StochasticDiffEq.jl} in the following.

\begin{figure}[p]
    \centering
    \includegraphics{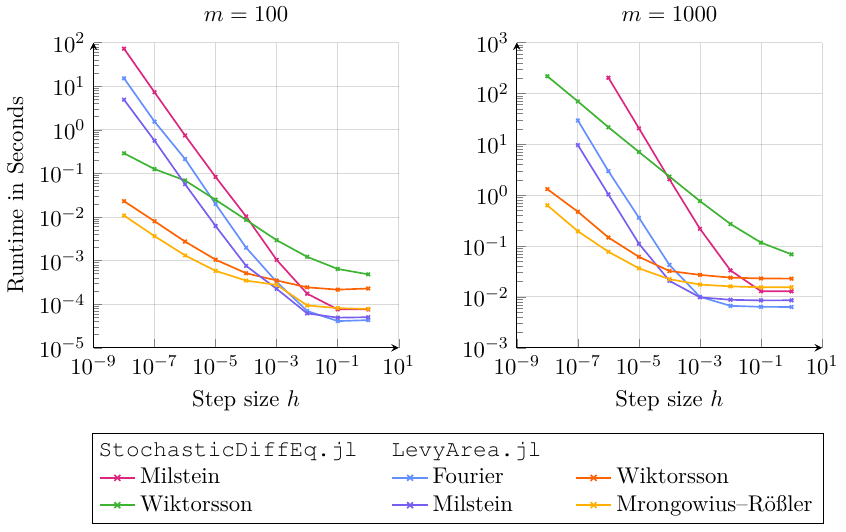}
    \caption{Runtime of the algorithms for variable step size and two different dimensions of the underlying Wiener process.
        The value of the truncation parameter is always chosen to guarantee that $ \lVert\mathcal{I}-\hat{\mathcal{I}}\rVert_\maxL \leq h^{3/2} $ according to the error estimates in Theorem~\ref{thm:approx-error-estimates}.
        The timings are averaged over 100 repetitions.
    }
    \label{fig:runtime-stepsize}
\end{figure}
\begin{table}[p]
    \renewcommand{\arraystretch}{1.3}
    \centering
    \begin{tabular}{@{}llll@{}}
        \toprule
        Algorithm \texttt{A} from & Algorithm \texttt{B} from & \multicolumn{2}{c}{Relative speed-up ($ \frac{t_\text{\texttt{A}}}{t_\text{\texttt{B}}} $)} \\
        \cmidrule{3-4}
        \package{StochasticDiffEq.jl} & \package{LevyArea.jl} & $ m=100 $ & $ m=1000 $ \\
        \midrule
        Milstein & Milstein & $ 1.5\times$--$14.8\times $ & $ 1.5\times$--$196.8\times $ \\
        Wiktorsson & Wiktorsson & $ 2.1\times$--$25.1\times $ & $ 3.0\times$--$165.8\times $ \\
        Wiktorsson & Mrongowius--Rößler & $ 6.2\times$--$52.1\times $ & $ 4.4\times$--$356.1\times $ \\
        \bottomrule
    \end{tabular}
    \caption{Minimal and maximal relative speed-ups of the new implementations in \package{LevyArea.jl} with respect to the implementations in \package{StochasticDiffEq.jl} for two different dimensions of the underlying Wiener process. See also Figure~\ref{fig:runtime-stepsize}.}
    \label{tab:speed-ups}
\end{table}

For the benchmark we measure the time
it takes to generate
the full matrix of iterated stochastic integrals over a range
of step sizes and for two different dimensions of the driving
Wiener process. To deal with measurement noise, we average
the computing times over 100 runs. Further, in order to
guarantee a fair comparison, we calculate the value of the truncation parameter
according to Table~\ref{tab:theoreticalcomparison} for all
algorithms under consideration.
We consider the setting where a strong order $ 1 $
numerical scheme is applied for the simulation of solutions
to SDEs. This setting is of high importance for many applications
and a typical situation where iterated stochastic integrals need
to be efficiently simulated. In order to retain the strong order $ 1 $
if the twofold iterated stochastic integrals are replaced by their
approximations, we need to
choose the precision for the simulated iterated stochastic
integrals following the discussion at the end of
Section~\ref{sec:error-estimates}.
Therefore, the error is always
chosen as $\bar{\varepsilon} = h^{3/2}$ in the \maxL-norm
and we consider step sizes $10^0, 10^{-1}, \ldots,
10^{-8}$ for the case $m=100$. For the case of $m=1000$, we
initially use the same step sizes, however we
refrain from applying the smaller step sizes for the Fourier
and Milstein algorithms whenever the computing time
or the needed memory becomes too large for reasonable
computations.

The Julia package \package{StochasticDiffEq.jl} is used
at version~6.37.1 and we run the in-place function
\jlinl{StochasticDiffEq.get_iterated_I!(h,W,nothing,iip,p)}
where \jlinl{iip} is a preallocated buffer.
Note that the creation of this buffer is not included in our
computing time measurements. The newly proposed Julia
package \package{LevyArea.jl} is used at its current
version~1.0.0.
The benchmark is performed on a computer with
an Intel Xeon E3-1245~v5 CPU at 3.50~GHz and 32~GB of memory
using Julia version~1.6.
Furthermore, the Julia package \package{DrWatson.jl} \cite{Datseris2020} is employed.
The simulation results for both settings with $m=100$ and
$m=1000$ are shown in Figure~\ref{fig:runtime-stepsize}.

Considering the benchmark results in
Figure~\ref{fig:runtime-stepsize}, we can see that
Wiktorsson's algorithm as well as the Mrongowius--R\"o{\ss}ler
algorithm attain a higher order of convergence compared
to the Fourier and the Milstein algorithms. This confirms the
theoretical results in Theorem~\ref{thm:approx-error-estimates}.
Moreover, the simulation results show that for the
Milstein algorithm as well as for Wiktorsson's algorithm
the implementation in the newly proposed package \package{LevyArea.jl}
(blue, purple, orange and yellow colours)
clearly outperforms the implementations in the package
\package{StochasticDiffEq.jl} (red and green colours). This is due to
the ideas for an efficient implementation presented in
Section~\ref{sec:algorithms} that are incorporated
in the package \package{LevyArea.jl}. This allows for a
speed-up by factors up to 165.8 for Wiktorsson's algorithm
in the case of $m=1000$ for the range of parameters we tested.

Comparing the cases $m=100$ and $m=1000$, it seems that
the implementations in the package \package{StochasticDiffEq.jl}
have a much higher overhead for high-dimensional Wiener processes.

Moreover, it can be seen that for both settings
the Mrongowius--R\"o{\ss}ler algorithm is the best algorithm
for sufficiently small step sizes as they typically
arise for SDE and SPDE approximation problems. This confirms
the theoretical results presented in Figure~\ref{fig:optimal-algorithm}.

The relative speed-up of the implementations in package
\package{LevyArea.jl} compared to package
\package{StochasticDiffEq.jl} are specified in
Table~\ref{tab:speed-ups}. There is a serious speed-up
for each algorithm in the package \package{LevyArea.jl} that
becomes even greater when step sizes are getting smaller or if
the dimension of the Wiener process is rather high. Thus, the
proposed software package \package{LevyArea.jl} allows for
very efficient simulations of iterated stochastic integrals
with very good performance also for small step sizes and
small error bounds as well as for high dimensional Wiener
processes.
This makes the package valuable especially for SDE and SPDE
simulations based on higher order approximation schemes.

{\emergencystretch=1em
\printbibliography[heading=bibintoc]
}

\end{document}